\documentclass[11pt]{article}
\pdfoutput=1

\usepackage[margin=1in]{geometry}

\usepackage{bm}
\usepackage{amsmath}
\usepackage{amsthm}
\usepackage{graphicx}
\usepackage[normalem]{ulem}
\usepackage{setspace}
\usepackage{amsthm}
\usepackage{amsfonts} 
\usepackage{epsfig}
\usepackage{hyphenat}
\usepackage{mathrsfs}
\usepackage{algorithm}
\usepackage{algorithmic}
\usepackage{tocvsec2}
\usepackage{caption}
\usepackage{makecell}
\usepackage{epsfig}%
\usepackage{subfigmat}  %{./subfigmat}
\usepackage{cases}%
\usepackage{amssymb}%
\usepackage{xspace}%
\usepackage{array}%

\usepackage{rotating}
\usepackage[table]{xcolor}
\usepackage{tabularx,colortbl}
\usepackage{float}
\usepackage{textcomp}%
\usepackage{caption}
\usepackage{booktabs, multicol, multirow}
\usepackage{amsthm}
\usepackage{dsfont}  
\newcommand{\norm}[1]{\left\lVert#1\right\rVert}

\theoremstyle{definition}
\newtheorem{theorem}{Theorem}

\newtheorem{proposition}{Proposition}
\newtheorem{lemma}{Lemma}

\newtheorem{remark}{Remark}
\newtheorem{example}{Example}

\usepackage{cite}
\usepackage{hyperref}

\usepackage[nameinlink]{cleveref}

\begin{document}
%--------------------------------------------------------------------

\title{An Optimal EDG Method for Distributed Control of Convection Diffusion PDEs}

\author{Xiao Zhang%
	\thanks{College of Mathematics, Sichuan University, China (zhangxiaofem@163.com). X.~Zhang thanks Missouri University of Science and Technology for hosting him as a visiting scholar; some of this work was completed during his research visit.}%
	\and
	Yangwen Zhang%
	\thanks{Department of Mathematics
		and Statistics, Missouri University of Science and Technology,
		Rolla, MO (\mbox{ywzfg4@mst.edu}, singlerj@mst.edu). J.~Singler and Y.~Zhang were supported in part by National Science Foundation grant DMS-1217122.  J.~Singler and Y.~Zhang thank the IMA for funding research visits, during which some of this work was completed.}
	\and
	John~R.~Singler%
	\footnotemark[2]
}

\maketitle

\begin{abstract}
	We propose an embedded discontinuous Galerkin (EDG) method to approximate the solution of a distributed control problem governed by convection diffusion PDEs, and obtain optimal a priori error estimates for the state, dual state, their fluxes, and the control. Moreover, we prove the optimize-then-discretize (OD) and discrtize-then-optimize (DO) approaches coincide. Numerical results confirm our theoretical results.
\end{abstract}

\section{Introduction}
\label{intro}
We study the following distributed optimal control problem: 
\begin{align}
\min J(u)=\frac{1}{2}\| y- y_{d}\|^2_{L^{2}(\Omega)}+\frac{\gamma}{2}\|u\|^2_{L^{2}(\Omega)}, \quad \gamma>0, \label{cost1}
\end{align}
subject to
\begin{equation}\label{Ori_problem}
\begin{split}
-\Delta y+\bm \beta\cdot\nabla y&=f+u \quad\text{in}~\Omega,\\
y&=g\qquad\quad\text{on}~\partial\Omega,
\end{split}
\end{equation}
where $\Omega\subset \mathbb{R}^{d} $ $ (d\geq 2)$ is a Lipschitz polyhedral domain  with boundary $\Gamma = \partial \Omega$, $ f \in L^2(\Omega) $, $ g \in C^0(\partial\Omega) $, and the vector field $\bm{\beta}$ satisfies
\begin{align}\label{beta_con}
\nabla\cdot\bm{\beta}  \le  0.
\end{align}

Optimal control problems for convection diffusion equations have been extensively studied using many different finite element methods, such as standard finite elements \cite{MR2851444,MR2719819,MR2475653}, mixed finite elements \cite{MR2550371,MR2851444,MR2971662}, discontinuous Galerkin (DG) methods \cite{MR2595051,MR2587414,MR2773301,MR3416418,MR3149415,MR3022208,MR2644299} and  hybrid discontinuous Galerkin (HDG) methods \cite{HuShenSinglerZhangZheng_HDG_Dirichlet_control2,HuShenSinglerZhangZheng_HDG_Dirichlet_control3}. HDG  methods were first introduced by Cockburn et al.\ in \cite{MR2485455} for second order elliptic problem, and then they have been applied to many other problems \cite{MR2772094,MR2513831,MR2558780, MR2796169, MR3626531,MR3522968,MR3477794,MR3463051,MR3452794,MR3343926}. HDG methods keep the advantages of DG methods, but have a lower number of globally coupled degrees of freedom compared to mixed methods and DG methods. However, the degrees of freedom for HDG methods is still larger compared to standard finite element methods. Embedded discontinuous Galerkin (EDG) methods were first proposed in \cite{MR2317378}, and then analyzed in \cite{MR2551142}. EDG methods are obtained from the HDG methods by forcing the numerical trace space to be continuous. This simple change significantly reduce the number of degrees of freedom and make EDG methods competitive for flow problems \cite{MR3404541} and many other applications \cite{peraire2011embedded,MR3404541,fernandez2016,MR3528316,fu2017analysis}.

In \cite{ZhangZhangSingler17}, we utilized an EDG method for a distributed optimal control problems for the Poisson equation.  We obtained optimal convergence rates for the state, dual state and the control, but \emph{suboptimal} convergence rates for their fluxes.  This suboptimal flux convergence rate for the Poisson equation is a limitation of the EDG method with equal order polynomial degrees for all variables \cite{MR2551142}.  However, Zhang, Xie, and Zhang recently proposed a new EDG method and proved optimal convergence rates for all variables for the Poisson equation \cite{ZhangXieZhang17}.  This EDG new method is obtained by simply using a lower degree finite element space for the flux.  In this work, we use this new EDG method to approximate the solution of the above convection diffusion distributed optimal control problem, and in \Cref{sec:analysis} we prove optimal convergence rates for all variables.% in \Cref{sec:analysis}. 

There are two main approaches to compute the numerical solution of PDE constrained optimal control problems: the optimize-then-discretize (OD) and discretize-then-optimize (DO) approaches. In the OD approach, one first derives the first-order necessary optimality conditions, then discretizes the optimality system, and then solves the resulting discrete system by utilizing efficient iterative solvers \cite{MR1990645}.  In the DO approach, one first discretizes the PDE optimization problem to obtain a finite dimensional optimization problem, which is then solved by existing optimization algorithms, such as \cite{betts2002practical,nocedal2006sequential}. The discretization methods for which these two approaches coincide are called \emph{commutative}.  Intuitively, the DO approach is more straightforward in practice; however, not all discretization schemes are commutative.  In the non-commutative case, the DO approach may result in badly behaved numerical results; see, e.g., \cite{JunLiuZhuWang,leykekhman2012investigation}. Therefore, devising {commutative} numerical methods is very important. In \Cref{sec:EDG}, we prove the EDG method studied here is commutative for the convection diffusion distributed control problem.  Moreover, we provide numerical examples to confirm our theoretical results in \Cref{sec:numerics}.

%--------------------------------------------------------
\section{EDG scheme for the optimal control problem}
\label{sec:EDG}
%--------------------------------------------------------

\subsection{Notation}

%To begin, we review some fundamental results concerning the optimality system for the control problem.

Throughout the paper we adopt the standard notation $W^{m,p}(\Omega)$ for Sobolev spaces on $\Omega$ with norm $\|\cdot\|_{m,p,\Omega}$ and seminorm $|\cdot|_{m,p,\Omega}$ . We denote $W^{m,2}(\Omega)$ by $H^{m}(\Omega)$ with norm $\|\cdot\|_{m,\Omega}$ and seminorm $|\cdot|_{m,\Omega}$. Specifically, $H_0^1(\Omega)=\{v\in H^1(\Omega):v=0 \;\mbox{on}\; \partial \Omega\}$.  We denote the $L^2$-inner products on $L^2(\Omega)$ and $L^2(\Gamma)$ by
\begin{align*}
(v,w) &= \int_{\Omega} vw  \quad \forall v,w\in L^2(\Omega),\\
\left\langle v,w\right\rangle &= \int_{\Gamma} vw  \quad\forall v,w\in L^2(\Gamma).
\end{align*}
Define the space $H(\text{div},\Omega)$ as
\begin{align*}
H(\text{div},\Omega) = \{\bm{v}\in [L^2(\Omega)]^d, \nabla\cdot \bm{v}\in L^2(\Omega)\}.
\end{align*}

Let $\mathcal{T}_h$ be a collection of disjoint elements that partition $\Omega$.  We denote by $\partial \mathcal{T}_h$ the set $\{\partial K: K\in \mathcal{T}_h\}$. For an element $K$ of the collection  $\mathcal{T}_h$, let $e = \partial K \cap \Gamma$ denote the boundary face of $ K $ if the $d-1$ Lebesgue measure of $e$ is non-zero. For two elements $K^+$ and $K^-$ of the collection $\mathcal{T}_h$, let $e = \partial K^+ \cap \partial K^-$ denote the interior face between $K^+$ and $K^-$ if the $d-1$ Lebesgue measure of $e$ is non-zero. Let $\varepsilon_h^o$ and $\varepsilon_h^{\partial}$ denote the set of interior and boundary faces, respectively. We denote by $\varepsilon_h$ the union of  $\varepsilon_h^o$ and $\varepsilon_h^{\partial}$. We finally introduce
\begin{align*}
(w,v)_{\mathcal{T}_h} = \sum_{K\in\mathcal{T}_h} (w,v)_K,   \quad\quad\quad\quad\left\langle \zeta,\rho\right\rangle_{\partial\mathcal{T}_h} = \sum_{K\in\mathcal{T}_h} \left\langle \zeta,\rho\right\rangle_{\partial K}.
\end{align*}

Let $\mathcal{P}^k(D)$ denote the set of polynomials of degree at most $k \geq 0$ on a domain $D$.  We introduce the discontinuous finite element spaces
\begin{align}
\bm{V}_h  &:= \{\bm{v}\in [L^2(\Omega)]^d: \bm{v}|_{K}\in [\mathcal{P}^k(K)]^d, \forall K\in \mathcal{T}_h\},\\
{W}_h  &:= \{{w}\in L^2(\Omega): {w}|_{K}\in \mathcal{P}^{k+1}(K), \forall K\in \mathcal{T}_h\},\\
{M}_h  &:= \{{\mu}\in L^2(\mathcal{\varepsilon}_h): {\mu}|_{e}\in \mathcal{P}^{k+1}(e), \forall e\in \varepsilon_h\}.
\end{align}
Define  $M_h(o)$ and $M_h(\partial)$ in the same way as $M_h$, but with $\varepsilon_h^o$ and $\varepsilon_h^\partial$ replacing $\varepsilon_h$.  Note that $M_h$ consists of functions which are continuous inside the faces (or edges) $e\in \varepsilon_h$ and discontinuous at their borders. In addition, for any function $w\in W_h$ we use $\nabla w$ to denote the piecewise gradient on each element $K\in \mathcal T_h$. A similar convention applies to the divergence $\nabla\cdot\bm r$ for all $\bm r\in \bm V_h$.

For EDG methods, we only change the space of numerical traces $M_h$, which is discontinuous, into a continuous space $\widetilde{M}_h$ as follows:
\begin{equation}
\widetilde{M}_h:=M_h \cap \mathcal{C}^0 (\varepsilon_h).
\end{equation}  
The spaces  $\widetilde{M}_h(o)$ and $\widetilde{M}_h(\partial)$ are defined in the same way as $M_h(o)$ and $M_h(\partial)$.

Recall we assume the Dirichlet boundary data $ g $ is continuous. Let $ \mathcal{I}_h $ be a interpolation operator, so that $\mathcal{I}_h g$ is a continuous interpolation of $g$ on $\varepsilon_h^\partial$.

Again, in most of the EDG works in the literature the polynomial degree is equal for the three spaces $ \bm V_h $, $ W_h $, and $ \widetilde{M}_h $.  We lower the polynomial degree for the flux space $ \bm V_h $ as in the recent work \cite{ZhangXieZhang17}.

%--------------------------------------------------------
\subsection{Optimize-then-Discretize}
%--------------------------------------------------------

First, we consider the optimize-then-discretize (OD) approach: we use the EDG method to discretize the optimality system for the convection diffusion control problem.

It is well known that the optimal control problem \eqref{cost1}-\eqref{Ori_problem} is equivalent to the optimality system
\begin{subequations}\label{eq_adeq}
	\begin{align}
	-\Delta y+\bm \beta\cdot\nabla y &=f+u  \qquad \text{in}~\Omega,\label{eq_adeq_a}\\
	y&=g\qquad\qquad\text{on}~\partial\Omega,\label{eq_adeq_b}\\
	-\Delta z-\nabla\cdot(\bm{\beta} z) &=y-y_d\qquad \text{in}~\Omega,\label{eq_adeq_c}\\
	z&=0\qquad\qquad\text{on}~\partial\Omega,\label{eq_adeq_d}\\
	z+\gamma  u&=0\qquad\qquad \ \text{in}~\Omega.\label{eq_adeq_e}
	\end{align}
\end{subequations}
For $\bm q = -\nabla y$ and $\bm p = -\nabla z$, the mixed weak form of the optimality system \eqref{eq_adeq_a}-\eqref{eq_adeq_e} is given by
\begin{subequations}\label{mixed}
	\begin{align}
	(\bm q,\bm r)-( y,\nabla\cdot \bm r)+\langle y,\bm r\cdot \bm n\rangle&=0,\label{mixed_a}\\
	(\nabla\cdot(\bm q+\bm \beta y),  w) - (y\nabla\cdot\bm \beta, w)&= ( f+ u, w),  \label{mixed_b}\\
	(\bm p,\bm r)-(z,\nabla \cdot\bm r)&=0,\label{mixed_c}\\
	(\nabla\cdot(\bm p-\bm \beta z),  w)&= (y- y_d, w),  \label{mixed_d}\\
	( z+\gamma u,v)&= 0,\label{mixed_e}
	\end{align}
\end{subequations}
for all $(\bm r, w,v)\in H(\text{div},\Omega)\times L^2(\Omega) \times L^2(\Omega)$.

To approximate the solution of this system, the EDG method seeks approximate fluxes ${\bm{q}}_h,{\bm{p}}_h \in \bm{V}_h $, states $ y_h, z_h \in W_h $, interior element boundary traces $ \widehat{y}_h^o,\widehat{z}_h^o \in \widetilde{M}_h(o) $, and  control $ u_h \in W_h$ satisfying
\begin{subequations}\label{EDG_discrete2}
	\begin{align}
	%%%%%%%%%%%%%
	(\bm q_h,\bm r_1)_{\mathcal T_h}-( y_h,\nabla\cdot\bm r_1)_{\mathcal T_h}+\langle \widehat y_h^o,\bm r_1\cdot\bm n\rangle_{\partial\mathcal T_h\backslash \varepsilon_h^\partial}&=-\langle  \mathcal I_hg,\bm r_1\cdot\bm n\rangle_{\varepsilon_h^\partial}, \label{EDG_discrete2_a}\\
	-(\bm q_h+\bm \beta y_h,  \nabla w_1)_{\mathcal T_h} - (y_h \nabla\cdot\bm \beta, w_1)_{\mathcal T_h}+\langle\widehat {\bm q}_h\cdot\bm n,w_1\rangle_{\partial\mathcal T_h}  \quad  \nonumber \\ 
	+\langle \bm \beta\cdot\bm n\widehat y_h^o,w_1\rangle_{\partial\mathcal T_h\backslash\varepsilon_h^\partial} -  ( u_h, w_1)_{\mathcal T_h} &=  - \langle \bm \beta\cdot\bm n \mathcal I_hg,w_1\rangle_{\varepsilon_h^\partial} \\
	&\quad + ( f, w_1)_{\mathcal T_h}, \label{EDG_discrete2_b}
	\end{align}
	for all $(\bm{r}_1, w_1)\in \bm{V}_h\times W_h$,
	\begin{align}
	(\bm p_h,\bm r_2)_{\mathcal T_h}-(z_h,\nabla\cdot\bm r_2)_{\mathcal T_h}+\langle \widehat z_h^o,\bm r_2\cdot\bm n\rangle_{\partial\mathcal T_h\backslash\varepsilon_h^\partial}&=0,\label{EDG_discrete2_c}\\
	-(\bm p_h-\bm \beta z_h, \nabla w_2)_{\mathcal T_h}+\langle\widehat{\bm p}_h\cdot\bm n,w_2\rangle_{\partial\mathcal T_h}  \quad  \nonumber\\
	-\langle\bm \beta\cdot\bm n\widehat z_h^o,w_2\rangle_{\partial\mathcal T_h\backslash \varepsilon_h^\partial} -  ( y_h, w_2)_{\mathcal T_h}&= -(y_d, w_2)_{\mathcal T_h},  \label{HDG_discrete2_d}
	\end{align}
	for all $(\bm{r}_2, w_2)\in \bm{V}_h\times W_h$,
	\begin{align}
	\langle\widehat {\bm q}_h\cdot\bm n+\bm \beta\cdot\bm n\widehat y_h^o,\mu_1\rangle_{\partial\mathcal T_h\backslash\varepsilon^{\partial}_h}&=0\label{EDG_discrete2_e},\\
	\langle\widehat{\bm p}_h\cdot\bm n-\bm \beta\cdot\bm n\widehat z_h^o,\mu_2\rangle_{\partial\mathcal T_h\backslash\varepsilon^{\partial}_h}&=0,\label{EDG_discrete2_f}
	\end{align}
	for all $\mu_1,\mu_2\in \widetilde{M}_h(o)$, and the optimality condition 
	\begin{align}
	(z_h+\gamma u_h, w_3 )_{\mathcal T_h} &=  0\label{EDG_discrete2_g},
	\end{align}
	for all $w_3\in W_h$.

	 The numerical traces on $\partial\mathcal{T}_h$ are defined as 
	\begin{align}
	\widehat{\bm{q}}_h\cdot \bm n &=\bm q_h\cdot\bm n+h^{-1} (y_h-\widehat y_h^o) +\tau_1 (y_h-\widehat y_h^o)   \qquad \mbox{on} \; \partial \mathcal{T}_h\backslash\varepsilon_h^\partial, \label{EDG_discrete2_h}\\
	\widehat{\bm{q}}_h\cdot \bm n &=\bm q_h\cdot\bm n+h^{-1} (y_h- \mathcal I_hg)+ \tau_1 (y_h- \mathcal I_hg)   \quad   \mbox{on}\;  \varepsilon_h^\partial, \label{HDG_discrete2_i}\\
	\widehat{\bm{p}}_h\cdot \bm n &=\bm p_h\cdot\bm n+h^{-1}(z_h-\widehat z_h^o)+ \tau_2(z_h-\widehat z_h^o)  \qquad  \; \mbox{on} \; \partial \mathcal{T}_h\backslash\varepsilon_h^\partial,\label{EDG_discrete2_j}\\
	\widehat{\bm{p}}_h\cdot \bm n &=\bm p_h\cdot\bm n+ h^{-1}z_h + \tau_2 z_h\qquad\qquad\qquad\qquad \; \mbox{on}\;  \varepsilon_h^\partial,\label{EDG_discrete2_k}
	\end{align}
\end{subequations}
where $\tau_1$ and $\tau_2$ are positive stabilization functions defined on $\partial \mathcal T_h$.  We show below that the OD and DO approaches coincide if $ \tau_2 = \tau_1 - \bm \beta \cdot \bm n $. The implementation of the OD approach is very similar to the HDG method in \cite{HuShenSinglerZhangZheng_HDG_Dirichlet_control2}, and hence is omitted here.

\subsection{Discretize-then-Optimize}
\label{DtO}
Now we derive the optimality conditions for the discretize-then-optimize (DO) approach when the optimal control problem is discretized by the EDG method. Therefore, we solve
\begin{equation}\label{cost_dis} 
%\min_{y_h\in W_h, u_h\in W_h}~ \frac{1}{2} \| y_h-y_d \|_{\mathcal{T}_h}^2+\frac{\gamma}{2}\| u_h \|_{\mathcal T_h}^2,\qquad \gamma>0,
\min_{u_h\in W_h}~ \frac{1}{2} \| y_h-y_d \|_{\mathcal{T}_h}^2+\frac{\gamma}{2}\| u_h \|_{\mathcal T_h}^2,\qquad \gamma>0,
\end{equation}
subject to the discrete state  equations
\begin{align}
(\bm q_h,\bm r_1)_{\mathcal T_h}-( y_h,\nabla\cdot\bm r_1)_{\mathcal T_h}+\langle \widehat y_h^o,\bm r_1\cdot\bm n\rangle_{\partial\mathcal T_h\backslash \varepsilon_h^\partial}=-\langle  \mathcal I_hg,\bm r_1\cdot\bm n\rangle_{\varepsilon_h^\partial}, &\label{EDG_discrete3_a}\\
-(\bm q_h+\bm \beta y_h,  \nabla w_1)_{\mathcal T_h} - (y_h \nabla\cdot\bm \beta, w_1)_{\mathcal T_h} +\langle {\bm q}_h\cdot\bm n, w_1\rangle_{\partial\mathcal T_h} \quad  \nonumber \\ 
+\langle (h^{-1}+\tau_1) y_h ,w_1\rangle_{\partial\mathcal T_h} +\langle  \bm \beta\cdot\bm n -(h^{-1}+\tau_1) y_h^o ,w_1\rangle_{\partial\mathcal T_h\backslash\varepsilon_h^\partial} \quad \nonumber \\ 
-  ( u_h, w_1)_{\mathcal T_h} = -\langle  \bm \beta\cdot\bm n -(h^{-1}+\tau_1) \mathcal I_h g ,w_1\rangle_{\partial\mathcal T_h\backslash\varepsilon_h^\partial}+ ( f, w_1)_{\mathcal T_h},& \label{EDG_discrete3_b}\\
\langle \bm q_h\cdot \bm n+(h^{-1}+\tau_1) ( y_h-\widehat y_h^o),\mu_1\rangle_{\partial\mathcal T_h\backslash\varepsilon^{\partial}_h}=0,&\label{EDG_discrete3_c}
\end{align}
for any $(\bm r_1,w_1,\mu_1)\in \bm V_h\times W_h\times \widetilde{M}_h(o)$.

The discretized Lagrangian functional is defined by
\begin{equation}\label{def_L}
\begin{split}
\hspace{1em}&\hspace{-1em} \mathcal L_h(\bm q_h,y_h,\widehat{y}_h^o;\bm p_h,z_h,\widehat{z}_h^o) = \frac{1}{2} \| y_h-y_d \|_{\mathcal{T}_h}^2+\frac{\gamma}{2}\| u_h \|_{\mathcal T_h}^2\\
& + (\bm q_h,\bm p_h)_{\mathcal T_h} - ( y_h,\nabla\cdot\bm p_h)_{\mathcal T_h}+\langle \widehat y_h^o,\bm p_h\cdot\bm n\rangle_{\partial\mathcal T_h\backslash \varepsilon_h^\partial}+\langle  \mathcal I_hg,\bm p_h\cdot\bm n\rangle_{\varepsilon_h^\partial}, \\
&+(\bm q_h+\bm \beta y_h,  \nabla z_h)_{\mathcal T_h} + (y_h \nabla\cdot\bm \beta, z_h)_{\mathcal T_h} - \langle {\bm q}_h\cdot\bm n, z_h\rangle_{\partial\mathcal T_h} \quad  \nonumber \\ 
&-\langle (h^{-1}+\tau_1) y_h ,z_h\rangle_{\partial\mathcal T_h} -\langle  (\bm \beta\cdot\bm n -h^{-1}-\tau_1) \widehat y_h^o ,z_h\rangle_{\partial\mathcal T_h\backslash\varepsilon_h^\partial} \quad \nonumber \\ 
&+ ( u_h, z_h)_{\mathcal T_h} - \langle  (\bm \beta\cdot\bm n -h^{-1}-\tau_1) \mathcal I_h g ,z_h\rangle_{\partial\mathcal T_h\backslash\varepsilon_h^\partial}+ ( f, z_h)_{\mathcal T_h},\\
&+\langle \bm q_h\cdot \bm n+(h^{-1}+\tau_1) ( y_h-\widehat y_h^o),\widehat z_h^o\rangle_{\partial\mathcal T_h\backslash\varepsilon^{\partial}_h}.
\end{split}
\end{equation}
Since the constraint  PDE is linear and the cost functional is convex, the necessary and sufficient optimality conditions can be obtained by setting the partial Fr\'echet-derivatives of \eqref{def_L} with respect to the flux $\bm q_h$, state $y_h$, numerical trace $\widehat y_h^o$ and control $u_h$ equal to zero. Thus, we obtain the system consisting of the discrete adjoint equations
\begin{subequations}
	\begin{align*}
	\frac{\partial \mathcal L_h}{\partial \bm q_h} \bm r_2&= (\bm p_h,\bm r_2)_{\mathcal{T}_h} + (\nabla z_h,\bm r_2)_{\mathcal{T}_h} - \langle  z_h, \bm r_2\cdot\bm n \rangle_{\partial \mathcal{T}_h} +\langle \widehat{z}_h^o, \bm r_2\cdot\bm n \rangle_{\partial \mathcal{T}_h\backslash\varepsilon_h^\partial}\nonumber\\
	& = (\bm p_h,\bm r_2)_{\mathcal{T}_h}-( z_h, \nabla\cdot\bm r_2)_{\mathcal{T}_h}+\langle \widehat{z}_h^o, \bm r_2\cdot\bm n \rangle_{\partial \mathcal{T}_h\backslash\varepsilon_h^\partial}=0,\\
	\frac{\partial \mathcal L_h}{\partial y_h} w_2&= -(\nabla\cdot \bm p_h,w_2)_{\mathcal{T}_h} + (\bm{\beta}  \nabla z_h, w_2)_{\mathcal T_h} +  (z_h\nabla\cdot\bm \beta,w_2 )_{\mathcal T_h} \\
	&\quad- \langle (h^{-1}+\tau_1) z_h, w_2\rangle_{\partial \mathcal T_h} +  \langle (h^{-1}+\tau_1) \widehat z_h^o, w_2\rangle_{\partial \mathcal T_h\backslash\varepsilon_h^\partial} +(y_h-y_d,w_2)_{\mathcal{T}_h} \\
	&= (\bm p_h - \bm \beta z_h,\nabla w_2)_{\mathcal{T}_h}  - \langle  \bm p_h\cdot\bm n + (h^{-1}+\tau_1 - \bm \beta \cdot\bm n ) z_h, w_2\rangle_{\partial \mathcal T_h}  \\
	&\quad+\langle (h^{-1}+\tau_1) \widehat z_h^o, w_2\rangle_{\partial \mathcal T_h\backslash\varepsilon_h^\partial} +(y_h-y_d,w_2)_{\mathcal{T}_h} =0,\\
	\frac{\partial \mathcal L_h}{\partial \widehat{y}_h^o}\mu_2&= \langle \bm p_h\cdot\bm n -(\bm{\beta}\cdot\bm n-h^{-1}-\tau_1) z_h - (h^{-1}+\tau_1) \widehat z_h^o, \mu_2\rangle_{\partial \mathcal{T}_h\backslash\varepsilon_h^\partial} = 0,
	\end{align*}
	Furthermore, we obtain the same optimality condition \eqref{EDG_discrete2_g} as in the OD approach:
	\begin{align*}
	\frac{\partial \mathcal L_h}{\partial u_h}  w_3= (\gamma u_h + z_h,w_3)_{\mathcal T_h}= 0.
	\end{align*}
\end{subequations}
In the OD approach, if the stabilization functions $ \tau_1 $ and $ \tau_2 $ satisfy
\begin{align}
\tau_2 = \tau_1 - \bm \beta \cdot \bm n,
\end{align}
then by comparing the above discrete adjoint equations with \eqref{EDG_discrete2} we obtain identical discrete systems; therefore, the two approaches coincide in this case, i.e., OD = DO.

\subsection{Implementation of DO}
\label{IDO}
In the DO approach, we need to deal with a large optimization problem \eqref{cost_dis} and \eqref{EDG_discrete3_a}-\eqref{EDG_discrete3_c} since the EDG method generates three variables: the flux $\bm q_h$, the scalar variable $y_h$, and the numerical trace $\widehat y_h$. Fortunately,  we can reduce the large scale problem into a smaller problem using the local solver for the EDG method.

\subsubsection{Matrix equations}

Assume $\bm{V}_h = \mbox{span}\{\bm\varphi_i\}_{i=1}^{N_1}$, $W_h=\mbox{span}\{\phi_i\}_{i=1}^{N_2}$, and $\widetilde{M}_h(o)=\mbox{span}\{\psi_i\}_{i=1}^{N_3} $. Then
\begin{equation}\label{expre}
\begin{split}
&\bm q_{h}= \sum_{j=1}^{N_1} \alpha_j \bm\varphi_j,  \quad
y_h = \sum_{j=1}^{N_2} \beta_j \phi_j, \quad \widehat{y}_h^o = \sum_{j=1}^{N_3}\gamma_{j}\psi_{j}, \quad u_h = \sum_{j=1}^{N_2} \zeta_{j} {\phi}_{j}.
\end{split}
\end{equation}
Substitute \eqref{expre} into \eqref{cost_dis}-\eqref{EDG_discrete3_c} to give the following finite dimensional  optimization problem:
\begin{subequations}\label{large_op}
	\begin{align}\label{large_op_a}
	%\min_{(\bm \beta,\bm\zeta) \in \mathbb R^{N_2} \times \mathbb R^{N_2}}  \frac 1 2 \bm \beta^T A_6 \bm \beta - b_1^T \bm \beta + \frac 1 2 \bm \zeta^T  A_6 \bm \zeta
	\min_{\bm\zeta \in \mathbb R^{N_2}}  \frac 1 2 \bm \beta^T A_6 \bm \beta - b_1^T \bm \beta + \frac 1 2 \bm \zeta^T  A_6 \bm \zeta
	\end{align}
	subject to 
	\begin{align}\label{large_op_b}
	\begin{bmatrix}
	A_1   &-A_2&A_3&0  \\
	A_2^T &A_4&A_5&-A_6 \\
	A_3^T&A_7 & -A_8 &0
	\end{bmatrix}
	\left[ {\begin{array}{*{20}{c}}
		\bm{\alpha}\\
		\bm{\beta}\\
		\bm{\gamma}\\
		\bm{\zeta}
		\end{array}} \right]
	=\left[ {\begin{array}{*{20}{c}}
		-b_2\\
		b_3-b_4\\
		0
		\end{array}} \right],
	\end{align}
\end{subequations}
where $\bm \alpha,   \bm \beta, \bm{\gamma},\bm \zeta$ are the coefficient vectors for $\bm q_h,y_h,\widehat y_h^o,u_h$, respectively, and
%
%where $\mathfrak{q},\mathfrak{p},\mathfrak{y},\mathfrak{z},\mathfrak{\widehat y},\mathfrak{\widehat z},\mathfrak{u}$ are the coefficient vectors for $\bm q_h,\bm p_h,y_h,z_h,\widehat y_h^o, \widehat z_h^o, u_h$ respectively,  and
%
\begin{gather*}
A_1= [(\bm\varphi_j,\bm\varphi_i )_{\mathcal{T}_h}],  \quad  A_2 = [(\phi_j,\nabla\cdot\bm{\varphi_i})_{\mathcal{T}_h}],  \quad  A_3 = [\langle \psi_j,\bm{\varphi}_i\cdot \bm n\rangle_{\partial\mathcal{T}_h\backslash \varepsilon_h^\partial}], \\
A_4 =- [(\phi_j,\nabla\cdot(\bm \beta\cdot \phi_i))_{\mathcal T_h}] +  [\langle (h^{-1} + \tau_1) \phi_j, \phi_i\rangle_{\partial \mathcal T_h}], \\
A_5 =  [\langle (\bm{\beta}\cdot\bm n- h^{-1} - \tau_1) \psi_j, \phi_i\rangle_{\partial \mathcal T_h\backslash\varepsilon_h^\partial}], \quad A_6= [(\phi_j,\phi_i )_{\mathcal{T}_h}], \\
A_7 = [\langle  (h^{-1}+\tau_1)\phi_j, \psi_i \rangle_{\partial{{\mathcal{T}_h}}\backslash\varepsilon_h^\partial}],  \quad  A_8 = [\langle  (h^{-1}+\tau_1)\psi_j, \psi_i \rangle_{\partial{{\mathcal{T}_h}}\backslash\varepsilon_h^\partial}], \\
b_1 = [(y_d,\phi_i )_{\mathcal{T}_h}], \quad b_2 = [\langle \mathcal I_h g, \bm r_1\cdot \bm n \rangle_{\varepsilon_h^\partial}], \quad  b_3 = [(f,\phi_i )_{\mathcal{T}_h}],\\
 b_4 = [\langle (\bm{\beta}\cdot\bm n- h^{-1} - \tau_1) g, \phi_i\rangle_{\varepsilon_h^\partial}].
\end{gather*}

Due to the discontinuous nature of the approximation spaces $\bm{V_h}$ and ${W_h}$, the first two equations of \eqref{large_op_b} can be used to eliminate both $\bm{\alpha}$ and $\bm{\beta}$ in an element-by-element fashion.  As a consequence, we can write system \eqref{large_op_b} as
\begin{align}\label{local_elei}
\left\{
\begin{aligned}
\bm \alpha  & =  G_1 \bm \gamma  + G_2 \bm{\zeta} + H_1,\\
\bm \beta  & =  G_3 \bm \gamma  + G_4 \bm{\zeta} + H_2,\\
G_5  &\bm \gamma + G_6 \bm \zeta = H_3.
\end{aligned}
\right.
\end{align}
% where 
% \begin{align*}
% G_5 = A_3^T G_1 + A_6G_3 - A_8, \quad  G_6 = A_3^T G_2 + A_6G_4, \quad  H_3 = -A_3^TH_1 - A_6H_2.
% \end{align*}
We provide details on the element-by-element construction of the coefficient matrices $G_1, \ldots, G_6 $ and $H_1, H_2, H_3$ in the appendix.

Substituting \eqref{local_elei} into \eqref{large_op} gives the reduced optimization problem
\begin{subequations}\label{reduced_op}
	\begin{align}\label{reduced_op_a}
%	\min_{(\bm \gamma,\bm\zeta) \in \mathbb R^{N_3} \times \mathbb R^{N_4}} \frac 1 2 \left[ {\begin{array}{*{20}{c}}
%		\bm \gamma^T  \  \bm{\zeta}^T
%		\end{array}} \right]
	\min_{\bm\zeta \in \mathbb R^{N_4}} \frac 1 2 \left[ {\begin{array}{*{20}{c}}
	\bm \gamma^T  \  \bm{\zeta}^T
	\end{array}} \right]
	\begin{bmatrix}
	B_1   & B_2  \\
	B_3  &   B_4 
	\end{bmatrix}
	\left[ {\begin{array}{*{20}{c}} 
		\bm \gamma \\ \bm{\zeta}
		\end{array}} \right] + \left[ {\begin{array}{*{20}{c}}
		b_5^T \   b_6^T
		\end{array}} \right]\left[ {\begin{array}{*{20}{c}} 
		\bm \gamma \\ \bm{\zeta}
		\end{array}} \right],
	\end{align}
	subject to 
	\begin{align}
	\left[ {\begin{array}{*{20}{c}}
		G_5 \ G_6
		\end{array}} \right]\left[ {\begin{array}{*{20}{c}} 
		\bm \gamma \\ \bm{\zeta}
		\end{array}} \right]  = H_3.
	\end{align}
\end{subequations}
where
\begin{gather*}
B_1 = G_3^TA_6G_3, \ \   B_2 = G_3^TA_6G_4, \ \  B_3 = G_4^TA_6G_3, \ \  B_4 = G_4^TA_6G_4 + A_6,\\
b_5 = G_3^T(A_6H_2 - b_1), \quad b_6 = G_4^T(A_6H_2 - b_1). 
\end{gather*}

\begin{remark}
	In the DO approach, we need to solve the optimization problem \eqref{reduced_op}; there are many existing optimization algorithms \cite{gill2005snopt} that can efficiently solve this problem.
\end{remark}

\section{Error Analysis}
\label{sec:analysis}

Next, we provide a convergence analysis of the above EDG method for the optimal control problem.  Throughout this section, we assume $ \bm \beta \in [W^{1,\infty}(\Omega)]^d $, $ \Omega $ is a bounded convex polyhedral domain, the solution is smooth enough, and $ h \leq 1 $.  

\subsection{Main result}
For our theoretical results, we require the stabilization functions $\tau_1$ and $\tau_2$ are chosen to satisfy
\begin{description}
	
	\item[\textbf{(A1)}] $\tau_2 = \tau_1 - \bm{\beta}\cdot \bm n$.
	
	\item[\textbf{(A2)}] For any  $K\in\mathcal T_h$, $\min{(\tau_1-\frac 1 2 \bm \beta \cdot \bm n)}|_{\partial K} >0$.
	
\end{description}
We note that \textbf{(A1)} and \textbf{(A2)} imply
\begin{equation}\label{eqn:tau1_condition}
\min{(\tau_2 + \frac 1 2 \bm \beta \cdot \bm n)}|_{\partial K} >0  \quad  \mbox{for any $K\in\mathcal T_h$.}
\end{equation}
Furthermore, \textbf{(A1)} implies the OD and DO approaches yield equivalent results; therefore, all of our convergence analysis is for the OD approach.

\begin{theorem}\label{main_res}
	We have
	\begin{align*}
	\|\bm q-\bm q_h\|_{\mathcal T_h}&\lesssim h^{k+1}(|\bm q|_{k+1}+|y|_{k+2}+|\bm p|_{k+1}+|z|_{k+2}),\\
	\|\bm p-\bm p_h\|_{\mathcal T_h}&\lesssim h^{k+1}(|\bm q|_{k+1}+|y|_{k+2}+|\bm p|_{k+1}+|z|_{k+2}),\\
	\|y-y_h\|_{\mathcal T_h}&\lesssim h^{k+2}(|\bm q|_{k+1}+|y|_{k+2}+|\bm p|_{k+1}+|z|_{k+2}),\\
	\|z-z_h\|_{\mathcal T_h}&\lesssim h^{k+2}(|\bm q|_{k+1}+|y|_{k+2}+|\bm p|_{k+1}+|z|_{k+2}),\\
	\|u-u_h\|_{\mathcal T_h}&\lesssim h^{k+2}(|\bm q|_{k+1}+|y|_{k+2}+|\bm p|_{k+1}+|z|_{k+2}).
	\end{align*}
\end{theorem}

\subsection{Preliminary material}
\label{sec:Projectionoperator}
Next, we  introduce the standard $L^2$-orthogonal projection operators $\bm{\Pi}_V$ and $\Pi_W$ as follows:
\begin{subequations} \label{def_L2}
	\begin{align}
	(\bm \Pi_V \bm q, \bm r)_{K}&=(\bm q,\bm r)_{K} \quad \forall \bm r\in [\mathcal{ P}_k(K)]^d,\\
	(\Pi_W y,w)_{K}&=(y,w)_{K}\quad \forall w\in \mathcal{P}_{k+1}(K).
	\end{align}
\end{subequations}
%
%We note the following classical results for $L^2$-orthogonal projection operator
We use the following well-known bounds:
\begin{subequations}\label{classical_ine}
	\begin{align}
	\norm {\bm q -\bm\Pi_V \bm q}_{\mathcal T_h} &\le  Ch^{k+1} \norm{\bm q}_{k+1,\Omega}, \ \, \norm {y -{\Pi_W y}}_{\mathcal T_h} \le C h^{k+2} \norm{y}_{k+2,\Omega},\\
	\norm {y -{\Pi_W y}}_{\partial\mathcal T_h} &\le C  h^{k+\frac 3 2} \norm{y}_{k+2,\Omega},
	\  \norm {\bm q -\bm\Pi_V \bm q}_{\partial \mathcal T_h} \le C h^{k+\frac 12} \norm{\bm q}_{k+1,\Omega},\\
	\norm {y -{ \mathcal I_h y}}_{\partial\mathcal T_h} &\le C  h^{k+\frac 3 2} \norm{y}_{k+2,\Omega}, \  \norm {w}_{\partial \mathcal T_h} \le C h^{-\frac 12} \norm {w}_{ \mathcal T_h}, \forall w\in W_h,
	\end{align}
\end{subequations}
where $\mathcal I_h $ is the continuous interpolation operator introduced earlier.%, and we have the same projection error bounds for $\bm p$ and $z$.

We define the following EDG operators $ \mathscr B_1$ and $ \mathscr B_2 $.
\begin{align}
\hspace{3em}&\hspace{-3em} \mathscr  B_1( \bm q_h,y_h,\widehat y_h^o;\bm r_1,w_1,\mu_1) \nonumber\\
&=(\bm q_h,\bm r_1)_{\mathcal T_h}-( y_h,\nabla\cdot\bm r_1)_{\mathcal T_h}+\langle \widehat y_h^o,\bm r_1\cdot\bm n\rangle_{\partial\mathcal T_h\backslash \varepsilon_h^\partial}\nonumber\\
&\quad-(\bm q_h+\bm \beta y_h,  \nabla w_1)_{\mathcal T_h}-(\nabla\cdot\bm\beta y_h,w_1)_{\mathcal T_h}\nonumber\\
&\quad+\langle {\bm q}_h\cdot\bm n +(h^{-1}+\tau_1)y_h, w_1\rangle_{\partial\mathcal T_h}+\langle (\bm\beta\cdot\bm n -h^{-1}-\tau_1) \widehat y_h^o,w_1\rangle_{\partial\mathcal T_h\backslash \varepsilon_h^\partial}\nonumber\\
&\quad -\langle  {\bm q}_h\cdot\bm n+\bm \beta\cdot\bm n\widehat y_h^o +(h^{-1}+\tau_1)(y_h-\widehat y_h^o), \mu_1\rangle_{\partial\mathcal T_h\backslash\varepsilon^{\partial}_h},\label{def_B1}\\
\hspace{3em}&\hspace{-3em} \mathscr B_2 (\bm p_h,z_h,\widehat z_h^o;\bm r_2, w_2,\mu_2)\nonumber\\
& =(\bm p_h,\bm r_2)_{\mathcal T_h}-( z_h,\nabla\cdot\bm r_2)_{\mathcal T_h}+\langle \widehat z_h^o,\bm r_2\cdot\bm n\rangle_{\partial\mathcal T_h\backslash\varepsilon_h^\partial}-(\bm p_h-\bm \beta z_h,  \nabla w_2)_{\mathcal T_h}\nonumber\\
&\quad+\langle {\bm p}_h\cdot\bm n +(h^{-1}+\tau_2) z_h, w_2\rangle_{\partial\mathcal T_h} -\langle (\bm \beta\cdot\bm n + h^{-1}+\tau_2)\widehat z_h^o ,w_2\rangle_{\partial\mathcal T_h\backslash\varepsilon_h^\partial}\nonumber\\
&\quad-\langle  {\bm p}_h\cdot\bm n-\bm \beta\cdot\bm n\widehat z_h^o +(h^{-1}+\tau_2)(z_h-\widehat z_h^o), \mu_2\rangle_{\partial\mathcal T_h\backslash\varepsilon^{\partial}_h}\label{def_B2}.
\end{align}

By the definition of $\mathscr B_1$ and $\mathscr B_2$,  we can rewrite the EDG formulation of the optimality system \eqref{EDG_discrete2} as follows: find $({\bm{q}}_h,{\bm{p}}_h,y_h,z_h,u_h,\widehat y_h^o,\widehat z_h^o)\in \bm{V}_h\times\bm{V}_h\times W_h \times W_h\times W_h\times \widetilde{M}_h(o)\times \widetilde{M}_h(o)$  such that
\begin{subequations}\label{EDG_full_discrete}
	\begin{align}
	\mathscr B_1(\bm q_h,y_h,\widehat y_h^o;\bm r_1,w_1,\mu_1)&=( f+ u_h, w_1)_{\mathcal T_h}\ \nonumber\\
	&\quad-\langle \mathcal I_hg, (\bm\beta\cdot\bm n-\tau_1 - h^{-1})w_1+\bm r_1\cdot\bm n \rangle_{\varepsilon_h^\partial},\label{EDG_full_discrete_a}\\
	\mathscr B_2(\bm p_h,z_h,\widehat z_h^o;\bm r_2,w_2,\mu_2)&=(y_h-y_d,w_2)_{\mathcal T_h},\label{EDG_full_discrete_b}\\
	(z_h+\gamma u_h,w_3)_{\mathcal T_h}&=  0,\label{EDG_full_discrete_e}
	\end{align}
\end{subequations}
for all $\left(\bm{r}_1, \bm{r}_2,w_1,w_2,w_3,\mu_1,\mu_2\right)\in \bm{V}_h\times\bm{V}_h\times W_h \times W_h\times W_h\times \widetilde{M}_h(o)\times \widetilde{M}_h(o)$.

Next, we present two fundamental properties of the operators $\mathscr B_1$ and $\mathscr B_2$,  and show the EDG equations \eqref{EDG_full_discrete} have a unique solution.  The proofs of these results are similar to proofs in \cite{HuShenSinglerZhangZheng_HDG_Dirichlet_control2,HuShenSinglerZhangZheng_HDG_Dirichlet_control3} and are omitted.  We note that condition \textbf{(A1)} is used in the proof of \Cref{identical_equa}, which is fundamental to the error analysis in this work.
\begin{lemma}\label{property_B}
	For any $ ( \bm v_h, w_h, \mu_h ) \in \bm V_h \times W_h \times\widetilde  M_h $, we have
	\begin{align*}
	\hspace{2em}&\hspace{-2em} \mathscr B_1(\bm v_h,w_h,\mu_h;\bm v_h,w_h,\mu_h)\\
	&=(\bm v_h,\bm v_h)_{\mathcal T_h}+ \langle (h^{-1}+\tau_1 - \frac 12 \bm \beta\cdot\bm n)(w_h-\mu_h),w_h-\mu_h\rangle_{\partial\mathcal T_h\backslash \varepsilon_h^\partial}\\
	&\quad-\frac 1 2(\nabla\cdot\bm\beta w_h,w_h)_{\mathcal T_h} +\langle (h^{-1}+\tau_1-\frac12\bm \beta\cdot\bm n) w_h,w_h\rangle_{\varepsilon_h^\partial},\\
	\hspace{2em}&\hspace{-2em}\mathscr B_2(\bm v_h,w_h,\mu_h;\bm v_h,w_h,\mu_h)\\
	&=(\bm v_h,\bm v_h)_{\mathcal T_h}+ \langle (h^{-1}+\tau_2 + \frac 12 \bm \beta\cdot\bm n)(w_h-\mu_h),w_h-\mu_h\rangle_{\partial\mathcal T_h\backslash \varepsilon_h^\partial}\\
	&\quad-\frac 1 2(\nabla\cdot\bm\beta w_h,w_h)_{\mathcal T_h} +\langle (h^{-1}+\tau_2+\frac12\bm \beta\cdot\bm n) w_h,w_h\rangle_{\varepsilon_h^\partial}.
	\end{align*}
\end{lemma}
\begin{lemma}\label{identical_equa}
	The EDG operators satisfy
	$$\mathscr B_1 (\bm q_h,y_h,\widehat y_h^o;\bm p_h,-z_h,-\widehat z_h^o) + \mathscr B_2 (\bm p_h,z_h,\widehat z_h^o;-\bm q_h,y_h,\widehat y_h^o) = 0.$$
\end{lemma}

\begin{proposition}\label{ex_uni}
	There exists a unique solution of the EDG equations \eqref{EDG_full_discrete}.
\end{proposition}
%
%\begin{proof}
%	To prove optimality system \eqref{EDG_full_discrete} has a unique solution, we just need to show that the optimization problem \eqref{cost_dis} and \eqref{EDG_discrete2} has a unique solution.
%	
%	By the idea of local solver, there exist a \emph{control to state} operator $\mathbb S: U_{ad}^h \to W_h$ such that $	y_h = \mathbb S  u_h $, substitue this into the  cost functional  \eqref{cost_dis} gives  the reduced cost functional
%	\begin{align*}
%	J_h(u_h) = \|\mathbb S u_h - y_d\|_{\mathcal T_h}^2  +  \frac{\gamma}{2} \|u_h\|_{\mathcal T_h}^2.
%	\end{align*}
%	Obviously,  $J_h$  is  \emph{continuous} and \emph{strictly convex} on $U_{ad}^h$. Moreover, as a bounded and closed set in a finite-dimensional space, $U_{ad}^h$ is compact. By the well-known Weierstrass theorem, $J_h$ attains its minimum in $U_{ad}^h$. Hence, there is some $u_h\in U_{ad}^h$  such that  $J_h( u_h) = \min_{\mu_h\in U_{ad}^h} J_h(\mu_h)$. The uniqueness of $u_h$ can be obtained by the fact that $J_h$ is strictly convex.
%\end{proof}

\subsection{Proof of Main Result}
To prove the convergence result, we split the proof into
six steps.  We first consider the following auxiliary problem: find $$({\bm{q}}_h(u),{\bm{p}}_h(u), y_h(u), z_h(u), {\widehat{y}}_h^o(u), {\widehat{z}}_h^o(u))\in \bm{V}_h\times\bm{V}_h\times W_h \times W_h\times \widetilde{M}_h(o)\times \widetilde{M}_h(o)$$ such that
\begin{subequations}\label{EDG_inter_u}
	\begin{align}
	\mathscr B_1(\bm q_h(u),y_h(u),\widehat{y}_h^o(u);\bm r_1, w_1,\mu_1)&=( f+ u,w_1)_{\mathcal T_h} \ \nonumber\\
	& \quad- \langle \mathcal I_hg, (\bm\beta\cdot\bm n-\tau_1 -h^{-1})w_1+\bm r_1\cdot\bm n \rangle_{\varepsilon_h^\partial},\label{EDG_u_a} \\
	\mathscr B_2(\bm p_h(u),z_h(u),\widehat{z}_h^o(u);\bm r_2, w_2,\mu_2)&=(y_h(u) - y_d, w_2)_{\mathcal T_h},\label{EDG_u_b}
	\end{align}
\end{subequations}
for all $\left(\bm{r}_1, \bm{r}_2,w_1,w_2,\mu_1,\mu_2\right)\in \bm{V}_h\times\bm{V}_h \times W_h\times W_h\times \widetilde{M}_h(o)\times \widetilde{M}_h(o)$.  

In Steps 1-3, we focus on the primary variables, i.e., the state $ y $ and the flux $ \bm{q} $, and we use the following notation:
\begin{equation}\label{notation}
\begin{split}
\delta^{\bm q} &=\bm q-{\bm\Pi}_V\bm q,  \qquad\qquad\quad \qquad\qquad\qquad\;\;\;\;\varepsilon^{\bm q}_h={\bm\Pi}_V \bm q-\bm q_h(u),\\
\delta^y&=y- \Pi_W y, \quad\qquad\qquad \qquad\qquad\qquad\;\;\;\; \;\varepsilon^{y}_h=\Pi_W y-y_h(u),\\
\delta^{\widehat y} &= y-\mathcal I_h y,  \qquad\qquad\qquad\quad\qquad\qquad\qquad  \varepsilon^{\widehat y}_h=\mathcal I_h y-\widehat{y}_h(u),\\
\widehat {\bm\delta}_1 &= \delta^{\bm q}\cdot\bm n+\bm \beta\cdot \bm n\delta^{\widehat{y}}+(\tau_1+h^{-1})  (\delta^y-\delta^{\widehat{y}}), 
\end{split}
\end{equation}
where $\widehat y_h(u) = \widehat y_h^o(u)$ on $\varepsilon_h^o$ and $\widehat y_h(u) = \mathcal I_h g$ on $\varepsilon_h^{\partial}$, which implies $\varepsilon_h^{\widehat y} = 0$ on $\varepsilon_h^{\partial}$.

\subsubsection{Step 1: The error equation for part 1 of the auxiliary problem \eqref{EDG_u_a}.} 
\begin{lemma} \label{lemma_error_y}
	We have the following error equation
	\begin{align} \label{error_y}
	\hspace{1em}\hspace{-1em} \mathscr B_1(\varepsilon^{\bm q}_h,\varepsilon^y_h,\varepsilon^{\widehat{y}}_h;\bm r_1, w_1,\mu_1) &=-\langle \delta^{\widehat{y}},\bm r_1\cdot\bm n \rangle_{\partial \mathcal{T}_h}+(\bm \beta\delta^y,\nabla w_1)_{\mathcal{T}_h}+(\nabla\cdot  \bm{\beta} \delta^y, w_1)_{\mathcal T_h}\nonumber\\
	&\quad-\langle\widehat{\bm \delta}_1,w_1\rangle_{\partial \mathcal{T}_h}+\langle \widehat{\bm \delta}_1,\mu_1 \rangle_{\partial \mathcal{T}_h\backslash\varepsilon_h^\partial}.
	\end{align}
\end{lemma}
\begin{proof}
	By definition of the operator $\mathscr B_1$ in \eqref{def_B1}, we have
	\begin{align*}
	\hspace{1em}&\hspace{-1em}  \mathscr  B_1( \bm \Pi_V \bm q,\Pi_W y,\mathcal I_h y;\bm r_1,w_1,\mu_1) \\
	&=( \bm \Pi_V \bm q,\bm r_1)_{\mathcal T_h}-( \Pi_W y,\nabla\cdot\bm r_1)_{\mathcal T_h}+\langle \mathcal I_h y,\bm r_1\cdot\bm n\rangle_{\partial\mathcal T_h\backslash \varepsilon_h^\partial} \\
	&\quad-( \bm \Pi_V \bm q+\bm \beta \Pi_W y,  \nabla w_1)_{\mathcal T_h} 	- (\nabla\cdot\bm{\beta} \Pi y,  w_1)_{{\mathcal{T}_h}} \\
	&\quad +\langle  \bm \Pi_V \bm q\cdot\bm n +(\tau_1+h^{-1}) \Pi_W y,w_1\rangle_{\partial\mathcal T_h} +\langle (\bm\beta\cdot\bm n -\tau_1 - h^{-1}) \mathcal I_h y,w_1\rangle_{\partial\mathcal T_h\backslash \varepsilon_h^\partial}\\
	&\quad -\langle   \bm \Pi_V \bm q\cdot\bm n+\bm \beta\cdot\bm n \mathcal I_h y +(\tau_1 + h^{-1})(\Pi_W y-\mathcal I_h y),\mu_1\rangle_{\partial\mathcal T_h\backslash\varepsilon^{\partial}_h}.
	\end{align*}
	Using properties of the $L^2$-orthogonal projection operators \eqref{def_L2} gives
	\begin{align*}
	\hspace{1em}&\hspace{-1em}  \mathscr  B_1( \bm \Pi_V \bm q,\Pi_W y, \mathcal I_h y;\bm r_1,w_1,\mu_1) \\
	&=( \bm q,\bm r_1)_{\mathcal T_h}-( y,\nabla\cdot\bm r_1)_{\mathcal T_h}+\langle  y,\bm r_1\cdot\bm n\rangle_{\partial\mathcal T_h\backslash \varepsilon_h^\partial}-\langle  \delta^{\widehat{y}},\bm r_1\cdot\bm n\rangle_{\partial\mathcal T_h\backslash \varepsilon_h^\partial} \\
	&\quad-( \bm q+\bm \beta y,  \nabla w_1)_{\mathcal T_h}+(\bm \beta \delta^y,\nabla w_1)_{\mathcal{T}_h} -(\nabla\cdot\bm{\beta} y,w_1)_{\mathcal T_h} + (\nabla\cdot\bm \beta \delta^y,w_1)_{\mathcal T_h} \\
	&\quad +\langle  \bm q\cdot\bm n +(\tau_1 +h^{-1})y,w_1\rangle_{\partial\mathcal T_h}-\langle  \delta^{\bm q}\cdot\bm n +(\tau_1 +h^{-1})  \delta^y,w_1\rangle_{\partial\mathcal T_h}\\
	&\quad+\langle (\bm\beta\cdot\bm n -\tau_1-h^{-1})  y,w_1\rangle_{\partial\mathcal T_h\backslash \varepsilon_h^\partial}-\langle (\bm\beta\cdot\bm n -\tau_1-h^{-1})  \delta^{\widehat{y}},w_1\rangle_{\partial\mathcal T_h\backslash \varepsilon_h^\partial}\\
	&\quad -\langle  \bm q\cdot\bm n+\bm \beta\cdot\bm n  y ,\mu_1\rangle_{\partial\mathcal T_h\backslash\varepsilon^{\partial}_h}\\
	&\quad+\langle \delta^{\bm q}\cdot\bm n+\bm \beta\cdot\bm n \delta^{\widehat{y}}+(\tau_1+h^{-1}) (\delta^y-\delta^{\widehat{y}}),\mu_1 \rangle_{\partial \mathcal{T}_h\backslash \varepsilon_h^\partial}.
	\end{align*}
	Note that the exact solution $\bm q$ and $y$ satisfies
	\begin{align*}
	(\bm q,\bm r_1)_{\mathcal{T}_h}-(y,\nabla \cdot \bm r_1)_{\mathcal{T}_h}+\langle y,\bm r_1\cdot \bm n \rangle_{\partial \mathcal{T}_h\backslash \varepsilon_h^\partial}&=-\langle  g,\bm r_1\cdot \bm n \rangle_{\varepsilon_h^\partial },\\
	-(\bm q+\bm \beta y,\nabla w_1)_{\mathcal{T}_h}-(\nabla\cdot\bm \beta y,w_1)_{\mathcal{T}_h}+\langle \bm q\cdot \bm n+\bm \beta\cdot \bm ny,w_1 \rangle_{\partial \mathcal{T}_h}&=(f+u,w_1)_{\mathcal{T}_h},\\
	-\langle \bm q\cdot \bm n+\bm \beta\cdot \bm n y,\mu_1 \rangle_{\partial \mathcal{T}_h\backslash \varepsilon_h^\partial}&=0,
	\end{align*}
	for all $(\bm r_1,w_1,\mu_1)\in \bm V_h\times W_h\times \widetilde{M}_h(o)$. Therefore, we have
	\begin{align*}
	\hspace{1em}&\hspace{-1em}  \mathscr  B_1( \bm \Pi_V \bm q,\Pi_W y,I_h y;\bm r_1,w_1,\mu_1)\\
	&= -\langle g,\bm r_1\cdot\bm n \rangle_{\varepsilon_h^\partial} -\langle  \delta^{\widehat{y}},\bm r_1\cdot\bm n\rangle_{\partial\mathcal T_h\backslash \varepsilon_h^\partial}+(\bm \beta\delta^y,\nabla w_1)_{\mathcal{T}_h}\\
	&\quad +(\nabla\cdot \bm \beta \delta^y,w_1)_{\mathcal{T}_h}+(f+u,w_1)_{\mathcal{T}_h}-\langle  \delta^{\bm q}\cdot\bm n ,w_1\rangle_{\partial\mathcal T_h}\\
	&\quad -\langle (\bm \beta\cdot \bm n-\tau_1-h^{-1}) y,w_1 \rangle_{\varepsilon_h^\partial}-\langle (\bm\beta\cdot\bm n -\tau_1-h^{-1})  \delta^{\widehat{y}},w_1\rangle_{\partial\mathcal T_h\backslash \varepsilon_h^\partial}\\
	&\quad +\langle \delta^{\bm q}\cdot\bm n+\bm \beta\cdot\bm n \delta^{\widehat{y}}+(\tau_1+h^{-1}) (\delta^y-\delta^{\widehat{y}}),\mu_1 \rangle_{\partial \mathcal{T}_h\backslash \varepsilon_h^\partial}.
	\end{align*}
	Finally, subtracting \eqref{EDG_u_a} from the above equation completes the proof.
\end{proof}

\subsubsection{Step 2: Estimate for $\varepsilon_h^q$ by an energy argument.}
\label{subsec:proof_step_2}
First, we give an auxiliary result that is very similar to a result from \cite{MR3440284}.  The proof is also very similar, and is omitted.
\begin{lemma} \label{error_nabla}
	We have 
	\begin{align}
	\| \nabla \varepsilon_h^y \|_{\mathcal{T}_h} \lesssim  \| \varepsilon_h^{\bm q} \|_{\mathcal{T}_h}+h^{-\frac{1}{2}} \| \varepsilon_h^y-\varepsilon_h^{\widehat{y}} \|_{\partial \mathcal{T}_h}.
	\end{align}
\end{lemma}
%\begin{proof}
%	Since $\nabla\varepsilon_h^y\in \bm V_h$, we take $(\bm r_1, w_1,\mu_1)=(\nabla \varepsilon_h^y,0,0)$ in \eqref{error_y} and integrate by parts to give
%	\begin{align*}
%	(\varepsilon_h^{\bm q}, \nabla \varepsilon_h^y)_{\mathcal T_h} + (\nabla \varepsilon_h^y, \nabla \varepsilon_h^y)_{\mathcal T_h}-\langle \varepsilon_h^y - \varepsilon_h^{\widehat y}, \nabla\varepsilon_h^y \cdot  \bm n\rangle_{\partial\mathcal T_h}=0.
%	\end{align*}
%	The Cauchy-Schwarz inequality and a trace inequality give \eqref{error_nabla}.
%\end{proof}
\begin{lemma} \label{energy_norm_q}
	We have
	\begin{align}
	\| \varepsilon_h^{\bm q} \|_{\mathcal{T}_h}+h^{-\frac{1}{2}} \| \varepsilon_h^y-\varepsilon_h^{\widehat{y}} \|_{\partial \mathcal{T}_h} \lesssim h^{k+1} (|\bm q|_{k+1}+|y|_{k+2}).
	\end{align}
\end{lemma}
\begin{proof}
	Taking $(\bm r_1,w_1,\mu_1)=(\varepsilon_h^{\bm q},\varepsilon_h^y,\varepsilon_h^{\widehat{y}})$ in \eqref{error_y} in  \Cref{lemma_error_y} gives
	\begin{align*}
	\hspace{1em}\hspace{-1em} \mathscr B_1(\varepsilon^{\bm q}_h,\varepsilon^y_h,\varepsilon^{\widehat{y}}_h;\varepsilon_h^{\bm q},\varepsilon_h^y,\varepsilon_h^{\widehat{y}}) &=-\langle \delta^{\widehat{y}},\varepsilon_h^{\bm q}\cdot\bm n \rangle_{\partial \mathcal{T}_h}+(\bm \beta\delta^y,\nabla \varepsilon_h^y)_{\mathcal{T}_h}\nonumber\\
	&\quad+(\nabla\cdot\bm \beta \delta^y,w_1)_{\mathcal{T}_h}-\langle\widehat{\bm \delta}_1,\varepsilon_h^y-\varepsilon_h^{\widehat{y}}\rangle_{\partial \mathcal{T}_h}\\
	&=: T_1+T_2+T_3+T_4,
	\end{align*}
	where we used $\varepsilon_h^{\widehat{y}}=0$ on $\varepsilon_h^\partial$. We estimate $T_i$, for $i=1,2,3,4$, as follows.  First,
	\begin{align*}
	T_1\le Ch^{-1} \| \delta^{\widehat{y}} \|_{\partial \mathcal{T}_h}^2+\frac{1}{4}\| \varepsilon_h^{\bm q} \|_{\mathcal{T}_h}^2,
	\end{align*}
	where we used trace and inverse inequalities. For the second term $T_2$, by  \Cref{error_nabla}, we have
	\begin{align*}
	T_2 &\le C \|\delta^y\|_{\mathcal{T}_h}^2+ \frac{1}{4}\| \varepsilon_h^{\bm q} \|_{\mathcal{T}_h}^2+\frac{1}{4h}\| \varepsilon_h^y-\varepsilon_h^{\widehat{y}} \|_{\partial \mathcal{T}_h}^2.
	\end{align*}
	For the third term $T_3$, we have
	\begin{align*}
		 T_3\le C\| \delta^y \|_{\mathcal{T}_h}^2+\frac{1}{2}\| (-\nabla\cdot\bm \beta)^{\frac{1}{2}}\varepsilon_h^y \|_{\mathcal{T}_h}^2.
	\end{align*}
	For the last term $T_4$, 
	\begin{align*}
	T_4 &\le Ch\| \widehat{\delta}_1 \|_{\partial \mathcal{T}_h}^2+\frac{1}{4h}\| \varepsilon_h^y-\varepsilon_h^{\widehat{y}} \|_{\partial \mathcal{T}_h}.
	\end{align*}
	
	Sum all the estimates for $\{ T_i \}_{i=1}^4$ to obtain
	\begin{align*}
	\| \varepsilon_h^{\bm q} \|_{\mathcal{T}_h}^2+h^{-1} \|\varepsilon_h^y-\varepsilon_h^{\widehat{y}} \|_{\partial \mathcal{T}_h}^2 &\lesssim h^{-1}\| \delta^{\widehat{y}} \|_{\partial \mathcal{T}_h}^2+\|\delta^y\|_{\mathcal{T}_h}^2+h\|\widehat{\delta}_1\|_{\partial \mathcal{T}_h}^2 \\
	&\lesssim h^{2(k+1)}(| \bm q |_{k+1}^2+|y|_{k+2}^2).
	\end{align*}
\end{proof}

\subsubsection{Step 3: Estimate for $\varepsilon_h^y$ by a duality argument.}
\label{subsec:proof_step_3}
Next, we introduce the dual problem for any given $\Theta$ in $L^2(\Omega)$:
\begin{equation}\label{Dual_PDE}
\begin{split}
\bm\Phi+\nabla\Psi&=0\qquad\qquad~\text{in}\ \Omega,\\
\nabla\cdot(\bm{\Phi}-\bm \beta \Psi)&=\Theta \qquad\text{in}\ \Omega,\\
\Psi&=0\qquad\qquad~\text{on}\ \partial\Omega.
\end{split}
\end{equation}
Since the domain $\Omega$ is convex, we have the following regularity estimate
\begin{align}\label{dual_esti}
\norm{\bm \Phi}_{1,\Omega} + \norm{\Psi}_{2,\Omega} \le C_{\text{reg}} \norm{\Theta}_\Omega.
\end{align}
We use the following notation below:%, which is similar to the earlier notation in \eqref{notation}:
\begin{align} 
\delta^{\bm \Phi} &=\bm \Phi-{\bm\Pi}_V\bm \Phi, \quad \delta^\Psi=\Psi- {\Pi}_W \Psi, \quad
\delta^{\widehat \Psi} = \Psi-\mathcal{I}_h \Psi.\label{notation_2}
\end{align}

\begin{lemma} \label{dual_y}
	We have
	\begin{align}
	\| \varepsilon_h^y \|_{\mathcal{T}_h} \lesssim h^{k+2} (|\bm q|_{k+1}+|y|_{k+2}).
	\end{align}
\end{lemma}
\begin{proof}
	First we take $(\bm r_1,w_1,\mu_1)=(\bm \Pi_V \bm \Phi, -\Pi_W \Psi,-\mathcal{I}_h \Psi)$ in equation \eqref{error_y} to get
	\begin{align*}
	\hspace{3em}&\hspace{-3em} \mathscr B_1(\varepsilon^{\bm q}_h,\varepsilon^y_h,\varepsilon^{\widehat{y}}_h;\bm \Pi_V \bm \Phi, -\Pi_W \Psi,-\mathcal{I}_h \Psi)\\
	&=(\varepsilon_h^{\bm q},\bm \Pi_V \bm \Phi)_{\mathcal T_h}-( \varepsilon_h^y,\nabla\cdot\bm \Pi_V \bm \Phi)_{\mathcal T_h}+\langle \varepsilon_h^{\widehat{y}},\bm \Pi_V \bm \Phi\cdot\bm n\rangle_{\partial\mathcal T_h\backslash \varepsilon_h^\partial} \\
	&\quad +(\varepsilon_h^{\bm q}+\bm \beta \varepsilon_h^y,  \nabla \Pi_W \Psi)_{\mathcal T_h}+(\nabla\cdot\bm \beta \varepsilon_h^y,\Pi_W \Psi)_{\mathcal{T}_h}\\
	&\quad-\langle \varepsilon_h^{\bm q}\cdot\bm n +(h^{-1}+\tau_1) \varepsilon_h^y,\Pi_W \Psi\rangle_{\partial\mathcal T_h}\\
	&\quad-\langle (\bm \beta \cdot \bm n-h^{-1}-\tau_1) \varepsilon_h^{\widehat{y}},\Pi_W \Psi \rangle_{\partial \mathcal{T}_h\backslash \varepsilon_h^\partial}\\
	&\quad+\langle  \varepsilon_h^{\bm q}\cdot\bm n+\bm \beta \cdot \bm n \varepsilon_h^{\widehat{y}}+(h^{-1}+\tau_1)(\varepsilon_h^y-\varepsilon_h^{\widehat{y}}),\mathcal{I}_h \Psi\rangle_{\partial\mathcal T_h\backslash\varepsilon^{\partial}_h}.
	\end{align*}
	Moreover, we have
	\begin{align*}
	-(\varepsilon_h^y,\nabla\cdot \bm \Pi_V \bm \Phi)_{\partial \mathcal{T}_h}&=(\nabla \varepsilon_h^y,\bm \Phi)_{\mathcal{T}_h}-\langle \varepsilon_h^y,\bm \Pi_V\Phi \cdot \bm n\rangle_{\partial \mathcal{T}_h}\\
	&=-(\varepsilon_h^y,\nabla\cdot \bm \Phi)_{\mathcal{T}_h}+\langle \varepsilon_h^y,\delta^{\bm \Phi} \cdot \bm n\rangle_{\partial \mathcal{T}_h},\\
%	\end{align*}
%	\begin{align*}
	( \varepsilon_h^{\bm q},\nabla \Pi_W \Psi)_{\mathcal{T}_h}&=-(\nabla\cdot \varepsilon_h^{\bm q}, \Psi)_{\mathcal{T}_h}+\langle \varepsilon_h^{\bm q}\cdot\bm n,\Pi_W \Psi \rangle_{\partial \mathcal{T}_h}\\
	&=(\varepsilon_h^{\bm q},\nabla \Psi)_{\mathcal{T}_h}-\langle \varepsilon_h^{\bm q}\cdot\bm n,\delta^\Psi \rangle_{\partial \mathcal{T}_h},\\
%	\end{align*}
%	\begin{align*}
	\quad(\bm \beta \varepsilon_h^y,\nabla \Pi_W \Psi)_{\mathcal{T}_h}+(\nabla\cdot\bm \beta \varepsilon_h^y,\Pi_W \Psi)_{\mathcal{T}_h} &= ( \varepsilon_h^y,\nabla\cdot(\bm \beta \Pi_W \Psi))_{\mathcal{T}_h}\\
	&=(\varepsilon_h^y,\nabla \cdot(\bm \beta \Psi))_{\mathcal{T}_h}-(\varepsilon_h^y,\nabla\cdot (\bm \beta\delta^\Psi))_{\mathcal{T}_h}\\
	&=(\varepsilon_h^y,\nabla \cdot(\bm \beta \Psi))_{\mathcal{T}_h}+(\bm \beta\cdot(\nabla\varepsilon_h^y),\delta^\Psi)_{\mathcal{T}_h}\\
	 & \quad -\langle \bm \beta\cdot\bm n \varepsilon_h^y,\delta^\Psi \rangle_{\partial \mathcal{T}_h}.
	\end{align*}
	Together with the dual problem \eqref{Dual_PDE}, using $\Theta=-\varepsilon_h^y$, we have
	\begin{align*}
	\hspace{3em}&\hspace{-3em} \mathscr B(\varepsilon^{\bm q}_h,\varepsilon^y_h,\varepsilon^{\widehat{y}}_h;\bm \Pi_V \bm \Phi, -\Pi_W \Psi,-\mathcal{I}_h \Psi)\\
	&=(\varepsilon_h^{\bm q}, \bm \Phi)_{\mathcal T_h}-( \varepsilon_h^y,\nabla\cdot \bm \Phi)_{\mathcal T_h}+\langle \varepsilon_h^y-\varepsilon_h^{\widehat{y}},\delta^{\bm \Phi}\cdot\bm n\rangle_{\partial\mathcal T_h} \\
	&\quad +(\varepsilon_h^{\bm q},  \nabla \Psi)_{\mathcal T_h}-\langle \varepsilon_h^{\bm q}\cdot\bm n,\delta^\Psi \rangle_{\partial \mathcal{T}_h}+(\varepsilon_h^y,\nabla\cdot (\bm \beta\Psi))_{\mathcal{T}_h}\\
	&\quad+(\bm \beta\cdot(\nabla\varepsilon_h^y),\delta^\Psi)_{\mathcal{T}_h}-\langle \bm \beta\cdot\bm n \varepsilon_h^y,\delta^\Psi \rangle_{\partial \mathcal{T}_h}\\
	&\quad-\langle (\tau_1+h^{-1}) (\varepsilon_h^y-\varepsilon_h^{\widehat{y}})+\bm \beta\cdot\bm n \varepsilon_h^{\widehat{y}},\Pi_W \Psi\rangle_{\partial\mathcal T_h}\\
	&\quad+\langle  \varepsilon_h^{\bm q}\cdot\bm n+(\tau_1+h^{-1})(\varepsilon_h^y-\varepsilon_h^{\widehat{y}})+\bm \beta\cdot\bm n \varepsilon_h^{\widehat{y}},\mathcal{I}_h \Psi\rangle_{\partial\mathcal T_h}\\
	&=(\varepsilon_h^y,\varepsilon_h^y)_{\mathcal{T}_h}+\langle \varepsilon_h^y-\varepsilon_h^{\widehat{y}},\delta^{\bm \Phi}\cdot \bm n \rangle_{\partial \mathcal{T}_h}-\langle \varepsilon_h^{\bm q}\cdot\bm n,\delta^{\widehat{\Psi}} \rangle_{\partial \mathcal{T}_h}\\
	&\quad +(\bm \beta\cdot\nabla\varepsilon_h^y,\delta^\Psi)_{\mathcal{T}_h}-\langle \bm \beta\cdot\bm n (\varepsilon_h^y-\varepsilon_h^{\widehat{y}}),\delta^\Psi \rangle_{\partial \mathcal{T}_h}\\
	&\quad+\langle (\tau_1+h^{-1}) (\varepsilon_h^y-\varepsilon_h^{\widehat{y}}),\delta^\Psi-\delta^{\widehat{\Psi}} \rangle_{\partial \mathcal{T}_h}.
	\end{align*}
	Here, we used that $\langle \varepsilon_h^{\widehat{y}},\bm \Phi\cdot \bm n \rangle_{\partial \mathcal{T}_h\backslash\varepsilon_h^\partial}=0$, $\Psi= \varepsilon_h^{\widehat{y}}=0$ on $\varepsilon_h^\partial$, and
	\begin{align*}
	\langle \bm \beta\cdot \bm n \varepsilon_h^{\widehat{y}},\delta^{\widehat{\Psi}} \rangle_{\partial \mathcal{T}_h}=0,
	\end{align*}
	since $\varepsilon_h^{\widehat y}$ is single-valued on interior faces and $\varepsilon_h^{\widehat y} = 0$ on boundary faces.
	On the other hand, from equation \eqref{error_y},
	\begin{align*}
	\hspace{3em}&\hspace{-3em}\mathscr B(\varepsilon^{\bm q}_h,\varepsilon^y_h,\varepsilon^{\widehat{y}}_h;\bm \Pi_V \bm \Phi, -\Pi_W \Psi,-\mathcal{I}_h \Psi) \\
&=-\langle \delta^{\widehat{y}},\bm \Pi_V \bm \Phi\cdot\bm n \rangle_{\partial \mathcal{T}_h}-(\bm \beta\delta^y,\nabla \Pi_W \Psi)_{\mathcal{T}_h}-(\nabla\cdot  \bm{\beta} \delta^y, \Pi_W \Psi)_{\mathcal T_h}\nonumber\\
&\quad+\langle\widehat{\bm \delta}_1,\Pi_W \Psi-\mathcal{I}_h \Psi\rangle_{\partial \mathcal{T}_h}.
	\end{align*}
	
	Comparing the two equations above, we have
	\begin{align*}
	\|\varepsilon_h^y\|_{\mathcal{T}_h}^2&=\langle \delta^{\widehat{y}},\delta^{\bm \Phi}\cdot\bm n \rangle_{\partial \mathcal{T}_h}-\langle\widehat{\bm \delta}_1,\delta^\Psi-\delta^{\widehat{\Psi}}\rangle_{\partial \mathcal{T}_h}-(\bm \beta\delta^y,\nabla\Pi_W \Psi)_{\mathcal{T}_h}\\
	&\quad -(\nabla\cdot\bm \beta \delta^y,\Pi_W \Psi)_{\mathcal{T}_h}-\langle \varepsilon_h^y-\varepsilon_h^{\widehat{y}},\delta^{\bm \Phi}\cdot \bm n \rangle_{\partial \mathcal{T}_h}+\langle \varepsilon_h^{\bm q}\cdot\bm n,\delta^{\widehat{\Psi}} \rangle_{\partial \mathcal{T}_h}\\
	&\quad -\langle (\tau_1+h^{-1}) (\varepsilon_h^y-\varepsilon_h^{\widehat{y}}),\delta^\Psi-\delta^{\widehat{\Psi}} \rangle_{\partial \mathcal{T}_h}+(\bm \beta\cdot(\nabla\varepsilon_h^y),\delta^\Psi)_{\mathcal{T}_h}\\
	&\quad-\langle \bm \beta\cdot\bm n (\varepsilon_h^y-\varepsilon_h^{\widehat{y}}),\delta^\Psi \rangle_{\partial \mathcal{T}_h}\\
	&=:\sum_{i=1}^9 T_i.
	\end{align*}
	We estimate each terms separately. For the first term
	\begin{align*}
	T_1 &\le \|\delta^{\widehat{y}}\|_{\partial \mathcal{T}_h} \|\delta^{\bm \Phi}\|_{\partial \mathcal{T}_h}\lesssim h^\frac{1}{2} \|\delta^{\widehat{y}}\|_{\partial \mathcal{T}_h}\|\bm \Phi\|_{1,\Omega}\lesssim h^\frac{1}{2} \|\delta^{\widehat{y}}\|_{\partial \mathcal{T}_h}\|\varepsilon_h^y\|_{\Omega}.
	\end{align*}
	For the second term, 
	\begin{align*}
	T_2 &\lesssim h^{\frac{3}{2}} \| \widehat{\delta}_1 \|_{\partial\mathcal{T}_h}\|\Psi\|_{2,\Omega}\lesssim h^{\frac{3}{2}} \| \widehat{\delta}_1 \|_{\partial\mathcal{T}_h}\| \varepsilon_h^y\|_{\mathcal{T}_h}.
	\end{align*}
	For the third term $T_3$,
	\begin{align*}
	T_3&\le \|\bm \beta\|_{0,\infty,\Omega}\| \delta^y \|_{\mathcal{T}_h}(\| \nabla \delta^\Psi \|_{\mathcal{T}_h}+\| \nabla \Psi \|_{\Omega})\\
	&\lesssim  \| \delta^y \|_{\mathcal{T}_h} (\| \Psi \|_{2,\Omega}+\| \Psi \|_{1,\Omega})\\
	&\lesssim  \| \delta^y \|_{\mathcal{T}_h} \| \varepsilon_h^y \|_{\mathcal{T}_h}.
	\end{align*}
	For $T_4$,
	\begin{align*}
		T_4 &\lesssim \| \bm \beta \|_{1,\infty,\Omega}\| \delta^y \|_{\mathcal{T}_h}\| \Pi_W \Psi \|_{\mathcal{T}_h}\lesssim \| \delta^y \|_{\mathcal{T}_h}\| \varepsilon_h^y \|_{\mathcal{T}_h}.
	\end{align*}
	For $T_5$,
	\begin{align*}
	T_5&\le \|\varepsilon_h^y-\varepsilon_h^{\widehat{y}}\|_{\partial \mathcal{T}_h}\| \delta^{\bm \Phi} \|_{\partial \mathcal{T}_h}\\
	&\lesssim h^\frac{1}{2} \|\varepsilon_h^y-\varepsilon_h^{\widehat{y}}\|_{\partial \mathcal{T}_h} \|\bm \Phi\|_{1,\Omega}\\
	&\lesssim h^\frac{1}{2} \|\varepsilon_h^y-\varepsilon_h^{\widehat{y}}\|_{\partial \mathcal{T}_h} \|\varepsilon_h^y\|_{\mathcal{T}_h}.
	\end{align*}
	For $T_6$, $T_7$, and $T_9$, following the same idea for $T_5$, we have
	\begin{align*}
	T_6 &\lesssim h\|\varepsilon_h^{\bm q}\|_{\mathcal{T}_h}\|\varepsilon_h^y\|_{\mathcal{T}_h},\\
	T_7 &\lesssim  h^\frac{1}{2} \| \varepsilon_h^y-\varepsilon_h^{\widehat{y}} \|_{\partial \mathcal{T}_h} \| \varepsilon_h^y \|_{\mathcal{T}_h},\\
	T_9 &\lesssim \|\bm \beta\|_{0,\infty,\Omega} h^{\frac{1}{2}}\| \varepsilon_h^y-\varepsilon_h^{\widehat{y}} \|_{\partial \mathcal{T}_h}\| \varepsilon_h^y \|_{\mathcal{T}_h}.
	\end{align*}
	And by Lemma \ref{error_nabla}, we have
	\begin{align*}
	T_8 &\lesssim\| \bm \beta \|_{0,\infty,\Omega}h\| \nabla \varepsilon_h^y \|_{\mathcal{T}_h}\|\Psi\|_{1}\\
	&\lesssim h(\| \varepsilon_h^{\bm q} \|_{\mathcal{T}_h}+h^{-\frac{1}{2}}\|\varepsilon_h^y-\varepsilon_h^{\widehat{y}}\|_{\mathcal{T}_h}  )\| \varepsilon_h^y \|_{\mathcal{T}_h}.
	\end{align*}
	Therefore, summing the estimates and using the bounds \eqref{classical_ine} and Lemma \ref{energy_norm_q} gives the result.	
\end{proof}
The triangle inequality yields optimal convergence rates for $\|\bm q -\bm q_h(u)\|_{\mathcal T_h}$ and $\|y -y_h(u)\|_{\mathcal T_h}$:

\begin{lemma}\label{le} We have
	\begin{subequations}
		\begin{align}
		\|\bm q -\bm q_h(u)\|_{\mathcal T_h} &\le \|\delta^{\bm q}\|_{\mathcal T_h} + \|\varepsilon_h^{\bm q}\|_{\mathcal T_h} \lesssim h^{k+1}(|\bm q|_{k+1}+|y|_{k+2}),\\
		\|y -y_h(u)\|_{\mathcal T_h} &\le \|\delta^{y}\|_{\mathcal T_h} + \|\varepsilon_h^{y}\|_{\mathcal T_h}  \lesssim  h^{k+2}(|\bm q|_{k+1}+|y|_{k+2}).
		\end{align}
	\end{subequations}
\end{lemma}

\subsubsection{Step 4: The error equation for part 2 of the auxiliary problem \eqref{EDG_u_b}.} \label{subsec:proof_step_4}

Next, we bound the error between the solution of the dual convection diffusion equation \eqref{eq_adeq_c}-\eqref{eq_adeq_d} for $ z $ and the auxiliary HDG equation \eqref{EDG_u_b}.   

First we define
\begin{equation}\label{notation_1}
\begin{split}
\delta^{\bm p} &=\bm p- {\bm\Pi}_V\bm p,  \qquad\qquad\qquad \qquad\qquad\qquad\;\;\;\;\varepsilon^{\bm p}_h={\bm\Pi}_V \bm p-\bm p_h(u),\\
\delta^z&=z- {\Pi}_W z, \qquad\qquad\qquad \qquad\qquad\qquad\;\;\;\; \;\varepsilon^{z}_h={\Pi}_W z-z_h(u),\\
\delta^{\widehat z} &= z-\mathcal I_hz,  \qquad\qquad\qquad\qquad\qquad\qquad\quad\;\;~~ \varepsilon^{\widehat z}_h= \mathcal I_h z-\widehat{z}_h(u),\\
\widehat {\bm\delta}_2 &= \delta^{\bm p}\cdot\bm n-\bm \beta\cdot\bm n \delta^{\widehat{z}} + (\tau_2+h^{-1})  (\delta^z-\delta^{\widehat{z}}),
\end{split}
\end{equation}
where $\widehat z_h(u) = \widehat z_h^o(u)$ on $\varepsilon_h^o$ and $\widehat z_h(u) = 0$ on $\varepsilon_h^{\partial}$.  This gives $\varepsilon_h^{\widehat z} = 0$ on $\varepsilon_h^{\partial}$.

Following the same idea with  \Cref{lemma_error_y}, we have the following error equation:
\begin{lemma}\label{lemma:step1_first_lemma}
	We have
	\begin{align}\label{error_z}
	\hspace{3em}&\hspace{-3em} \mathscr B_2(\varepsilon^{\bm p}_h,\varepsilon^z_h,\varepsilon^{\widehat{z}}_h;\bm r_2, w_2,\mu_2) \ \nonumber\\
	&=-\langle \delta^{\widehat{z}},\bm r_2\cdot\bm n \rangle_{\partial \mathcal{T}_h}-(\bm \beta\delta^z,\nabla w_2)_{\mathcal{T}_h}\nonumber\\
	&\quad-\langle\widehat{\bm \delta}_2,w_2\rangle_{\partial \mathcal{T}_h}+\langle \widehat{\bm \delta}_2,\mu_2 \rangle_{\partial \mathcal{T}_h\backslash\varepsilon_h^\partial}+(y_h(u)-y,w_2)_{\mathcal{T}_h}.
	\end{align}
\end{lemma}

\subsubsection{Step 5: Estimates for $\varepsilon_h^p$ and $\varepsilon_h^z$ by an energy and duality argument.} \label{subsec:proof_step3}

First, it is easy to see that  \Cref{error_nabla} still holds for $\varepsilon_h^z$, $\varepsilon_h^{\widehat{z}}$, and $\varepsilon_h^{\bm p}$.
\begin{lemma}
	We have
	\begin{align}
	\| \nabla \varepsilon_h^z \|_{\mathcal{T}_h}\le C( \| \varepsilon_h^{\bm q} \|_{\mathcal{T}_h}+h^{-\frac{1}{2}}\| \varepsilon_h^z-\varepsilon_h^{\widehat{z}} \|_{\partial \mathcal{T}_h} ).\label{energy_z}
	\end{align}
\end{lemma}
Also, to estimate $\varepsilon_h^{\bm p}$ we need the following discrete Poincar\'e inequality that is very similar to a result from \cite{MR3440284}.  The proof is essentially the same, and is omitted.
\begin{lemma} 
   We have
   \begin{align}
   	\| \varepsilon_h^z \|_{\mathcal{T}_h} \le C(\| \nabla\varepsilon_h^z \|_{\mathcal{T}_h}+h^{-\frac{1}{2}}\| \varepsilon_h^z-\varepsilon_h^{\widehat{z}} \|_{\partial \mathcal{T}_h}).\label{poincare}
   \end{align}
\end{lemma}
%\begin{proof}
%	By the discrete Poincar{\'e} inequality from \cite{MR1974504}, we have 
%	\begin{equation}\label{poin_in}
%	\begin{split}
%	\|\varepsilon_h^z\|_{\mathcal T_h} &\le C(\|\nabla \varepsilon_h^z\|_{\mathcal T_h}+h^{-\frac 1 2} \|[\varepsilon_h^z]\|_{\varepsilon_h})\\
%	%
%	%
%	%
%	&= C(\|\nabla \varepsilon_h^z\|_{\mathcal T_h}+h^{-\frac 1 2} \|[\varepsilon_h^z - \varepsilon_h^{\widehat z}]\|_{\varepsilon_h})\\
%	%
%	%
%	%
%	&\le  C(\|\nabla \varepsilon_h^z\|_{\mathcal T_h}+h^{-\frac 1 2} \|\varepsilon_h^z - \varepsilon_h^{\widehat z}\|_{\partial \mathcal T_h}).
%	\end{split}
%	\end{equation}
%	Here, $[\varepsilon_h^z]$ equals the jump of $\varepsilon_h^z$ between adjacent interior elements and $[\varepsilon_h^z] = \varepsilon_h^z$ on $\varepsilon_h^\partial$. Therefore, the above identity holds since $\varepsilon_h^{\widehat z} $ is single-valued on interior faces and zero on $\varepsilon_h^\partial$, and the last inequality holds due to the triangle inequality.
%%	Here, $[\varepsilon_h^z]_{\varepsilon_h^o}$ is the jump of $\varepsilon_h^z$ between adjacent elements and $[\varepsilon_h^z] = \varepsilon_h^z$ on $\varepsilon_h^\partial$. Therefore the above identity holds since $\varepsilon_h^{\widehat z} $ is single valued on interior face and zero on $\varepsilon_h^\partial$, and the last inequality  due to the triangle inequality.
%\end{proof}

\begin{lemma}\label{e_sec}
	We have
	\begin{align}
	\|\varepsilon_h^{\bm p}\|_{\mathcal T_h}+h^{-\frac{1}{2}}\|\varepsilon_h^z-\varepsilon_h^{\widehat z}\|_{\partial\mathcal T_h} &\lesssim h^{k+1}(|\bm q|_{k+1}+|y|_{k+2}+|\bm p|_{k+1}+|z|_{k+2}),\label{error_es_p}\\
		\|\varepsilon_h^{ z}\|_{\mathcal T_h}& \lesssim h^{k+1}(|\bm q|_{k+1}+|y|_{k+2}+|\bm p|_{k+1}+|z|_{k+2}).\label{error_es_z_inter}
	\end{align}
\end{lemma}
\begin{proof}
	Since $\varepsilon_h^{\widehat{z}}=0$ on $\varepsilon_h^{\partial}$, the energy identity for $\mathscr B_2$ in  \Cref{property_B} gives
	\begin{align*}
		\mathscr B_2&(\varepsilon_h^{\bm p},\varepsilon_h^z,\varepsilon_h^{\widehat{z}},\varepsilon_h^{\bm p},\varepsilon_h^z,\varepsilon_h^{\widehat{z}})\\
		&=\|\varepsilon_h^{\bm p}\|_{\mathcal{T}_h}^2+\| (h^{-1}+\tau_2+\frac{1}{2}\bm \beta\cdot\bm n)^{\frac{1}{2}} (\varepsilon_h^z-\varepsilon_h^{\widehat{z}}) \|_{\partial \mathcal{T}_h}^2+\frac{1}{2}\| (-\nabla\bm \beta)^{\frac{1}{2}}\varepsilon_h^z \|_{\mathcal{T}_h}^2.
	\end{align*}
	Take $(\bm r_2,w_2,\mu_2)=(\varepsilon_h^{\bm p},\varepsilon_h^z,\varepsilon_h^{\widehat{z}})$ in the error equation \eqref{error_z} to obtain
	\begin{align*}
		\|\varepsilon_h^{\bm p}\|_{\mathcal{T}_h}^2&+\| (h^{-1}+\tau_2+\frac{1}{2}\bm \beta\cdot\bm n)^{\frac{1}{2}} (\varepsilon_h^z-\varepsilon_h^{\widehat{z}}) \|_{\partial \mathcal{T}_h}^2+\frac{1}{2}\| (-\nabla\bm \beta)^{\frac{1}{2}}\varepsilon_h^z \|_{\mathcal{T}_h}^2\\
		&=-\langle \delta^{\widehat{z}},\bm \varepsilon_h^{\bm p}\cdot\bm n \rangle_{\partial \mathcal{T}_h}-(\bm \beta\delta^z,\nabla \varepsilon
		_h^z)_{\mathcal{T}_h}\nonumber\\
		&\quad-\langle\widehat{\bm \delta}_2,\varepsilon_h^z-\varepsilon_h^{\widehat{z}}\rangle_{\partial \mathcal{T}_h}+(y_h(u)-y,\varepsilon_h^z)_{\mathcal{T}_h}\\
		&=: T_1+T_2+T_3+T_4.
	\end{align*}
	By the same argument as in the proof of  \Cref{energy_norm_q}, apply \eqref{energy_z} and \eqref{poincare} to get
	\begin{align*}
		T_1 &\lesssim h^{-\frac{1}{2}}\| \delta^{\widehat{z}} \|_{\partial \mathcal{T}_h} \| \varepsilon_h^{\bm p} \|_{\mathcal{T}_h},\\
		T_2 &\lesssim \| \bm \beta \|_{0,\infty,\Omega} \| \delta^z\|_{\mathcal{T}_h}\| \nabla\varepsilon_h^z \|_{\mathcal{T}_h}\\
		&\lesssim \| \bm \beta \|_{0,\infty,\Omega} \| \delta^z\|_{\mathcal{T}_h}(\| \varepsilon_h^{\bm p} \|_{\mathcal{T}_h}+h^{-\frac{1}{2}}\| \varepsilon_h^z-\varepsilon_h^{\widehat{z}} \|_{\partial \mathcal{T}_h}),\\
		T_3 &\lesssim h^{\frac{1}{2}}\| \widehat{\bm \delta}_2 \|_{\partial \mathcal{T}_h}h^{-\frac{1}{2}}\| \varepsilon_h^z-\varepsilon_h^{\widehat{z}} \|_{\partial \mathcal{T}_h},\\
		T_4 &\lesssim \| y-y_h(u) \|_{\mathcal{T}_h}\| \varepsilon_h^z \|_{\mathcal{T}_h}\\
		&\lesssim \| y-y_h(u) \|_{\mathcal{T}_h} (\| \varepsilon_h^{\bm p} \|_{\mathcal{T}_h}+h^{-\frac{1}{2}}\| \varepsilon_h^z-\varepsilon_h^{\widehat{z}} \|_{\partial \mathcal{T}_h}).
	\end{align*}
	Finally, applying \eqref{classical_ine} and  \Cref{le} yields \eqref{error_es_p}. Together with \eqref{error_es_p} and \eqref{poincare}, we can obtain \eqref{error_es_z_inter}.
\end{proof}

\subsubsection{Step 6: Estimate for $\varepsilon_h^z$ by a duality argument.}

Next, we introduce the dual problem for any given $\Theta$ in $L^2(\Omega)$:
\begin{equation}\label{Dual_PDE2}
\begin{split}
\bm\Phi+\nabla\Psi&=0\qquad\qquad~\text{in}\ \Omega,\\
\nabla\cdot\bm{\Phi}-\bm \beta\cdot\nabla \Psi&=\Theta \qquad\text{in}\ \Omega,\\
\Psi&=0\qquad \ \text{on}\ \partial\Omega.
\end{split}
\end{equation}
Since the domain $\Omega$ is convex, we have the following regularity estimate
\begin{align}\label{dual_esti2}
\norm{\bm \Phi}_{1,\Omega} + \norm{\Psi}_{2,\Omega} \le C_{\text{reg}} \norm{\Theta}_\Omega.
\end{align}

\begin{lemma}
	We have
	\begin{subequations}
		\begin{align}
		\|\varepsilon^z_h\|_{\mathcal T_h} &\lesssim h^{k+2}(|\bm q|_{k+1}+|y|_{k+2}+|\bm p|_{k+1}+|z|_{k+2}).\label{var_z}
		\end{align}
	\end{subequations}
\end{lemma}

\begin{proof}
	Consider the dual problem \eqref{Dual_PDE2}, and let $\Theta = \varepsilon_h^z$.  Take  $(\bm r_2,w_2,\mu_2) = ( {\bm\Pi}_V\bm{\Phi},-{\Pi}_W\Psi,-I_h\Psi)$ in \eqref{error_z} in  \Cref{lemma:step1_first_lemma}.  Since $\Psi=0$ on $\varepsilon_h^{\partial}$ we have
	\begin{align*}
	%%%%%%%%%%%%%%%%%%%%%%%%%%%%%%%
	\hspace{1em}&\hspace{-1em}  \mathscr B_2 (\varepsilon^{\bm p}_h,\varepsilon^z_h,\varepsilon^{\widehat z}_h;{\bm\Pi}_V\bm{\Phi},-{\Pi}_W\Psi,-I_h\Psi)\\
	%%%%%%%%%%%%%%%%%%%%%%%%%%%%%%%
	&= (\varepsilon^{\bm p}_h,{\bm\Pi}_V\bm{\Phi})_{\mathcal T_h}-( \varepsilon^z_h,\nabla\cdot{\bm\Pi}_V\bm{\Phi})_{\mathcal T_h}+\langle \varepsilon^{\widehat z}_h,{\bm\Pi}_V\bm{\Phi} \cdot\bm n\rangle_{\partial\mathcal T_h}\\
	%%%%%%%%%%%%%%%%%%%%%%%%%%%%%%%%
	&\quad+(\varepsilon^{\bm p}_h-\bm \beta \varepsilon^z_h,  \nabla {\Pi}_W\Psi)_{\mathcal T_h}
	-\langle  \varepsilon^{\bm p}_h\cdot\bm n-\bm \beta\cdot\bm n\varepsilon^{\widehat z}_h +\tau_2(\varepsilon^z_h-\varepsilon^{\widehat z}_h ), {\Pi}_W\Psi -\mathcal{I}_h \Psi\rangle_{\partial\mathcal T_h}.	
	%%%%%%%%%%%%%%%%%%%%%%%%%%%%%%%%%
	\end{align*}	
	Moreover, we have
	\begin{align*}
	-(\varepsilon_h^z,\nabla\cdot \bm \Pi_V \bm \Phi)_{\partial \mathcal{T}_h}&=(\nabla \varepsilon_h^z,\bm \Phi)_{\mathcal{T}_h}-\langle \varepsilon_h^z,\bm \Pi_V\bm\Phi \cdot \bm n\rangle_{\partial \mathcal{T}_h}\\
	&=-(\varepsilon_h^z,\nabla\cdot \bm \Phi)_{\mathcal{T}_h}+\langle \varepsilon_h^z,\delta^{\bm \Phi} \cdot \bm n\rangle_{\partial \mathcal{T}_h},
	\end{align*}
	\begin{align*}
	( \varepsilon_h^{\bm p},\nabla \Pi_W \Psi)_{\mathcal{T}_h}&=-(\nabla\cdot \varepsilon_h^{\bm p}, \Psi)_{\mathcal{T}_h}+\langle \varepsilon_h^{\bm p}\cdot\bm n,\Pi_W \Psi \rangle_{\partial \mathcal{T}_h}\\
	&=(\varepsilon_h^{\bm p},\nabla \Psi)_{\mathcal{T}_h}-\langle \varepsilon_h^{\bm p}\cdot\bm n,\delta^\Psi \rangle_{\partial \mathcal{T}_h},
	\end{align*}
	\begin{align*}
	-(\bm \beta \varepsilon_h^z,\nabla \Pi_W \Psi)_{\mathcal{T}_h}&=-(\bm \beta\varepsilon_h^z,\nabla\delta^{\Psi})_{\mathcal{T}_h}+(\bm \beta\varepsilon_h^z,\nabla {\Psi})_{\mathcal{T}_h}\\
	&=-\langle \bm \beta\cdot\bm n \varepsilon_h^z,\delta^{\Psi} \rangle_{\partial \mathcal{T}_h}+(\nabla\cdot \bm \beta \varepsilon_h^z,\delta^{\Psi})_{\mathcal{T}_h}\\
	&\quad +(\bm \beta \cdot
	 \nabla \varepsilon_h^z,\delta^{\Psi})_{\mathcal{T}_h}+(\bm \beta\varepsilon_h^z,\nabla {\Psi})_{\mathcal{T}_h}.
	\end{align*}
	Then we have
	\begin{align*}
	%%%%%%%%%%%%%%%%%%%%%%%%%%%%%%%
	\hspace{1em}&\hspace{-1em}  \mathscr B_2 (\varepsilon^{\bm p}_h,\varepsilon^z_h,\varepsilon^{\widehat z}_h;{\bm\Pi}_V\bm{\Phi},-{\Pi}_W\Psi,-I_h\Psi)\\
	%%%%%%%%%%%%%%%%%%%%%%%%%%%%%%%
	&= (\varepsilon^{\bm p}_h, \bm{\Phi})_{\mathcal T_h}-(\varepsilon_h^z,\nabla\cdot \bm \Phi)_{\mathcal{T}_h}+\langle \varepsilon_h^z,\delta^{\bm \Phi} \cdot \bm n\rangle_{\partial \mathcal{T}_h}+\langle \varepsilon^{\widehat z}_h,{\bm\Pi}_V\bm{\Phi} \cdot\bm n\rangle_{\partial\mathcal T_h}\\
	&\quad +(\varepsilon_h^{\bm p},\nabla \Psi)_{\mathcal{T}_h}-\langle \varepsilon_h^{\bm p}\cdot\bm n,\delta^{\widehat{\Psi}} \rangle_{\partial \mathcal{T}_h} -\langle \bm \beta\cdot\bm n \varepsilon_h^z,\delta^{\Psi} \rangle_{\partial \mathcal{T}_h}\\
	%%%%%%%%%%%%%%%%%%%%%%%%%%%%%%%%\
	&\quad +(\nabla\cdot\bm \beta\varepsilon_h^z,\delta^\Psi)_{\mathcal{T}_h}+(\bm \beta\cdot\nabla\varepsilon_h^z,\delta^\Psi)_{\mathcal{T}_h}+(\varepsilon_h^z,\bm \beta\cdot \nabla\Psi)_{\mathcal{T}_h}\\
	%%%%%%%%%%%
	&\quad +\langle \bm \beta\cdot\bm n \varepsilon_h^{\widehat{z}},\delta^\Psi \rangle_{\partial \mathcal{T}_h}
	+\langle  \tau_2(\varepsilon^z_h-\varepsilon^{\widehat z}_h ), \delta^\Psi-\delta^{\widehat{\Psi}}\rangle_{\partial\mathcal T_h}\\
	&= (\varepsilon_h^z,\varepsilon_h^z)_{\mathcal{T}_h}+\langle \varepsilon_h^z-\varepsilon_h^{\widehat{z}},\delta^{\bm \Phi}\cdot\bm n \rangle_{\partial \mathcal{T}_h}-\langle \varepsilon_h^{\bm p}\cdot\bm n,\delta^{\widehat{\Psi}} \rangle_{\partial \mathcal{T}_h}+(\nabla\cdot\bm \beta \varepsilon_h^z,\delta^\Psi)_{\mathcal{T}_h}\\
	&\quad +(\bm \beta\cdot\nabla\varepsilon_h^z,\delta^\Psi)_{\mathcal{T}_h} -\langle \bm \beta\cdot\bm n (\varepsilon_h^z-\varepsilon_h^{\widehat{z}}),\delta^\Psi \rangle_{\partial \mathcal{T}_h}\\
	&\quad+\langle  (\tau_2+h^{-1})(\varepsilon^z_h-\varepsilon^{\widehat z}_h ), \delta^\Psi-\delta^{\widehat{\Psi}}\rangle_{\partial\mathcal T_h}.
	%%%%%%%%%%%%%%%%%%%%%%%%%%%%%%%%%
	\end{align*}
	Here, we used $\langle\varepsilon^{\widehat z}_h,\bm \Phi\cdot\bm n\rangle_{\partial\mathcal T_h}=0$, which holds since  $\varepsilon^{\widehat z}_h$ is single-valued function on interior edges and $\varepsilon^{\widehat z}_h=0$ on $\varepsilon^{\partial}_h$.  We also used $ \langle \bm \beta\cdot \bm n \varepsilon_h^{\widehat{z}},\delta^{\widehat{\Psi}} \rangle_{\partial \mathcal{T}_h}=0 $, which is derived similarly.
	%
%	Furthermore,
%	\begin{align*}
%	\langle \bm \beta\cdot \bm n \varepsilon_h^{\widehat{z}},\delta^{\widehat{z}} \rangle_{\partial \mathcal{T}_h}=0.
%	\end{align*} 
	
	On the other hand, by  \Cref{lemma:step1_first_lemma}
	\begin{align}
	\hspace{3em}&\hspace{-3em} \mathscr B_2(\varepsilon^{\bm p}_h,\varepsilon^z_h,\varepsilon^{\widehat{z}}_h;\bm \Pi_V \bm \Phi,-\Pi_V \Psi,-\mathcal{I}_h \Psi) \ \nonumber\\
	&=-\langle \delta^{\widehat{z}},\bm \Pi_V \bm \Phi\cdot\bm n \rangle_{\partial \mathcal{T}_h}+(\bm \beta\delta^z,\nabla \Pi_V \Psi)_{\mathcal{T}_h}\nonumber\\
	&\quad+\langle\widehat{\bm \delta}_2,\Pi_V \Psi-\mathcal{I}_h \Psi\rangle_{\partial \mathcal{T}_h}-(y_h(u)-y,\Pi_V \Psi)_{\mathcal{T}_h}.
	\end{align}
	Comparing the above two equalities gives
	\begin{align*}
	%%%%%%%%%%%%%%%%%%%%%%%%%%%%%%%%%%%%
	\|  \varepsilon_h^z\|_{\mathcal T_h}^2 & = -\langle \varepsilon_h^z-\varepsilon_h^{\widehat{z}},\delta^{\bm \Phi}\cdot\bm n+(\tau_2+h^{-1})(\delta^\Psi-\delta^{\widehat{\Psi}})-\bm \beta\cdot\bm n \delta^\Psi \rangle_{\partial \mathcal{T}_h}\\
	&\quad +\langle \varepsilon_h^{\bm p}\cdot\bm n,\delta^{\widehat{\Psi}} \rangle_{\partial \mathcal{T}_h}+(\bm \beta\cdot\nabla\varepsilon_h^z,\delta^\Psi)_{\mathcal{T}_h}+(\nabla\cdot\bm \beta\varepsilon_h^z,\delta^\Psi)_{\mathcal{T}_h}\\
	%%%%%%%%%%%%%%%%%%%%%%%%%%%%%%%%%%%%
	%%%%%%%%%%%%%%%%%%%%%%%%%%%
	&\quad-\langle \delta^{\widehat{z}},\delta^{\bm \Phi}\cdot\bm n \rangle_{\partial \mathcal{T}_h}+(\bm \beta\delta^z,\nabla \Pi_V \Psi)_{\mathcal{T}_h}\nonumber\\
	&\quad+\langle\widehat{\bm \delta}_2,\Pi_V \Psi-\mathcal{I}_h \Psi\rangle_{\partial \mathcal{T}_h}-(y_h(u)-y,\Pi_V \Psi)_{\mathcal{T}_h}\\
	&=:  \sum_{i=1}^8 R_i.
	\end{align*}
 For the terms $R_1$-$R_4$,  \Cref{e_sec} gives
	\begin{align*}
	%%%%%%%%%%%%%%%%%%%%%%%%%%%%%%
	R_1 &= -\langle \varepsilon^z_h - \varepsilon^{\widehat z}_h,\delta^{\bm \Phi}\cdot\bm n - \bm{\beta}\cdot\bm n\delta^{\Psi}+(\tau_2+h^{-1})(\delta^{\Psi} - \delta^{\widehat \Psi})\rangle_{\partial\mathcal T_h}\\
	%%%%%%%%%%%%%%%%%%%%%%%%%%%%%
	&\lesssim h^\frac{1}{2}\|(\tau_2+h^{-1}+\bm \beta\cdot\bm n)^{\frac{1}{2}} (\varepsilon_h^z-\varepsilon_h^{\widehat{z}} )\|_{\partial \mathcal{T}_h} (\| \bm \Phi \|_{1,\Omega}+\| \Psi \|_{1,\Omega}) \\
	&\lesssim h^\frac{1}{2}\|(\tau_2+h^{-1}+\bm \beta\cdot\bm n)^{\frac{1}{2}} (\varepsilon_h^z-\varepsilon_h^{\widehat{z}} )\|_{\partial \mathcal{T}_h} \| \varepsilon_h^z \|_{\mathcal{T}_h},\\
	%%%%%%%%%%%%%%%%%%%%%%%%%%%%%%
	R_2&\lesssim h^\frac{3}{2} \| \varepsilon_h^{\bm p} \|_{\partial \mathcal{T}_h}\| \Psi \|_{2,\Omega}\lesssim h^\frac{3}{2} \| \varepsilon_h^{\bm p} \|_{\partial \mathcal{T}_h}\| \varepsilon_h^z \|_{\mathcal{T}_h},\\
	R_3&\lesssim \| \bm \beta \|_{0,\infty,\Omega}h\| \nabla \varepsilon_h^z \|_{\mathcal{T}_h}\| \Psi \|_{1,\Omega},\\
	R_4&\lesssim h\| (-\nabla\cdot\bm \beta)^{\frac{1}{2}} \varepsilon_h^z \|_{\mathcal{T}_h}\| \Psi \|_{1,\Omega}\lesssim h\| (-\nabla\cdot\bm \beta)^{\frac{1}{2}} \varepsilon_h^z \|_{\mathcal{T}_h}\| \varepsilon_h^z \|_{\mathcal{T}_h}.
	\end{align*}
	For $R_5$, we have
	\begin{align*}
	R_5\lesssim  h^\frac{1}{2}\| \delta^{\widehat{z}} \|_{\partial \mathcal{T}_h}\| \varepsilon_h^z \|_{\mathcal{T}_h}.
	\end{align*}
	For the terms $R_6$ and $R_8$, we use the triangle inequality, the regularity estimate \eqref{dual_esti}, and the assumption $ h \leq 1 $ to give
	\begin{align*}
	R_6&\lesssim \| \bm \beta \|_{0,\infty,\Omega}\| \delta^z \|_{\mathcal{T}_h}(\|\nabla \delta^{\Psi}\|_{\mathcal{T}_h}+\| \Psi \|_{\mathcal{T}_h})\lesssim \| \bm \beta \|_{0,\infty,\Omega}\| \delta^z \|_{\mathcal{T}_h}\| \varepsilon_h^z \|_{\mathcal{T}_h},\\
	R_8&\lesssim \| y_h(u)-y \|_{\mathcal{T}_h}\| \varepsilon_h^z \|_{\mathcal{T}_h}.
	\end{align*}
	For the term $R_7$, 
	\begin{align*}
	R_7&\lesssim h^\frac{3}{2} \| \delta^{\bm p}\cdot\bm n+(\tau_1 +h^{-1} )(\delta^z-\delta^{\widehat{z}}) \|_{\partial \mathcal{T}_h}\| \Psi\|_{2,\Omega}\\
	&\lesssim h^\frac{3}{2}(\|\delta^{\bm p}\|_{\partial\mathcal{T}_h}+\|\delta^z \|_{\mathcal{T}_h}+\| \delta^{\widehat{z}} \|_{\partial \mathcal{T}_h})\| \varepsilon_h^z \|_{\mathcal{T}_h}.
	\end{align*}
	
	Summing $R_1$ to $R_8$, together with \eqref{classical_ine}, \eqref{energy_z}, \eqref{error_es_p}, and \eqref{error_es_z_inter} gives
	\begin{align*}
	\|\varepsilon^z_h\|_{\mathcal T_h}\lesssim h^{k+2}(|\bm q|_{k+1}+|y|_{k+2}+|\bm p|_{k+1}+|z|_{k+2}).
	\end{align*}
\end{proof}
The triangle inequality gives optimal convergence rates for $\|\bm p -\bm p_h(u)\|_{\mathcal T_h}$ and $\|z -z_h(u)\|_{\mathcal T_h}$:

\begin{lemma}\label{lemma:step3_conv_rates}
	\begin{subequations}
		\begin{align}
		\|\bm p -\bm p_h(u)\|_{\mathcal T_h} &\le \|\delta^{\bm p}\|_{\mathcal T_h} + \|\varepsilon_h^{\bm p}\|_{\mathcal T_h} \ \nonumber \\ 
		&\lesssim h^{k+1}(|\bm q|_{k+1}+|y|_{k+2}+|\bm p|_{k+1}+|z|_{k+2}),\\
		\|z -z_h(u)\|_{\mathcal T_h} &\le \|\delta^{z}\|_{\mathcal T_h} + \|\varepsilon_h^{z}\|_{\mathcal T_h} \ \nonumber \\
		& \lesssim  h^{k+2}(|\bm q|_{k+1}+|y|_{k+2}+|\bm p|_{k+1}+|z|_{k+2}).
		\end{align}
	\end{subequations}
\end{lemma}

\subsubsection{Step 7: Estimates for $\|u-u_h\|_{\mathcal T_h}$, $\norm {y-y_h}_{\mathcal T_h}$, and $\norm {z-z_h}_{\mathcal T_h}$.}

Next, we bound the error between the solutions of the auxiliary problem and the EDG discretization of the optimality system \eqref{EDG_full_discrete}.  We use these error bounds and the error bounds in  \Cref{le} and \Cref{lemma:step3_conv_rates} to obtain the main result.

The proofs in Steps 7 and 8 are similar to proofs in our earlier work \cite{HuShenSinglerZhangZheng_HDG_Dirichlet_control2}; we include the proofs here to make the final steps self-contained.

For the remaining steps, we denote 
\begin{equation*}
\begin{split}
\zeta_{\bm q} &=\bm q_h(u)-\bm q_h,\quad\zeta_{y} = y_h(u)-y_h,\quad\zeta_{\widehat y} = \widehat y_h^o(u)-\widehat y_h^o,\\
\zeta_{\bm p} &=\bm p_h(u)-\bm p_h,\quad\zeta_{z} = z_h(u)-z_h,\quad\zeta_{\widehat z} = \widehat z_h^o(u)-\widehat z_h^o.
\end{split}
\end{equation*}
%where $\widehat y_h =\widehat y_h^o $, $\widehat z_h =\widehat z_h^o $on $\varepsilon_h^o$ and $\widehat y_h = \mathcal I_h g$, $\widehat z_h = 0$ on $\varepsilon_h^{\partial}$,  which implies $\varepsilon_h^{\widehat y} =\varepsilon_h^{\widehat z} = 0$ on $\varepsilon_h^{\partial}$.
%
Subtracting the auxiliary problem and the EDG problem gives the following error equations
\begin{subequations}\label{eq_yh}
	\begin{align}
	\mathscr B_1(\zeta_{\bm q},\zeta_y,\zeta_{\widehat y};\bm r_1, w_1,\mu_1)&=(u-u_h,w_1)_{\mathcal T_h}\label{eq_yh_yhu},\\
	\mathscr B_2(\zeta_{\bm p},\zeta_z,\zeta_{\widehat z};\bm r_2, w_2,\mu_2)&=-(\zeta_y, w_2)_{\mathcal T_h}\label{eq_zh_zhu}.
	\end{align}
\end{subequations}
\begin{lemma}
	We have
	\begin{align}\label{eq_uuh_yhuyh}
	\hspace{3em}&\hspace{-3em}  \gamma\|u-u_h\|^2_{\mathcal T_h}+\|y_h(u)-y_h\|^2_{\mathcal T_h}\nonumber\\
	&=( z_h+\gamma u_h,u-u_h)_{\mathcal T_h}-(z_h(u)+\gamma u,u-u_h)_{\mathcal T_h}.
	\end{align}
\end{lemma}
\begin{proof}
	First, we have
	\begin{align*}
	\hspace{3em}&\hspace{-3em}  ( z_h+\gamma u_h,u-u_h)_{\mathcal T_h}-( z_h(u)+\gamma u,u-u_h)_{\mathcal T_h}\\
	&=-(\zeta_{ z},u-u_h)_{\mathcal T_h}+\gamma\|u-u_h\|^2_{\mathcal T_h}.
	\end{align*}
	Next,  \Cref{identical_equa} gives
	\begin{align*}
	\mathscr B_1 &(\zeta_{\bm q},\zeta_y,\zeta_{\widehat{y}};\zeta_{\bm p},-\zeta_{z},-\zeta_{\widehat z}) + \mathscr B_2(\zeta_{\bm p},\zeta_z,\zeta_{\widehat z};-\zeta_{\bm q},\zeta_y,\zeta_{\widehat{y}})  = 0.
	\end{align*}
	On the other hand, using the definition of $ \mathscr B_1 $ and $ \mathscr B_2 $ gives
	\begin{align*}
	\hspace{3em}&\hspace{-3em}  \mathscr B_1 (\zeta_{\bm q},\zeta_y,\zeta_{\widehat{y}};\zeta_{\bm p},-\zeta_{z},-\zeta_{\widehat z}) + \mathscr B_2(\zeta_{\bm p},\zeta_z,\zeta_{\widehat z};-\zeta_{\bm q},\zeta_y,\zeta_{\widehat{y}})\\
	&= - ( u- u_h,\zeta_{ z})_{\mathcal{T}_h}-\|\zeta_{ y}\|^2_{\mathcal{T}_h}.
	\end{align*}
	Comparing the above two equalities gives
	\begin{align*}
	-(u-u_h,\zeta_{ z})_{\mathcal{T}_h}=\|\zeta_{ y}\|^2_{\mathcal{T}_h}.
	\end{align*}
	This completes the proof.
\end{proof}

\begin{theorem}\label{thm:estimates_u_y_z}
	We have
	\begin{subequations}
		\begin{align}\label{err_yhu_yh}
		\|u-u_h\|_{\mathcal T_h}&\lesssim h^{k+2}(|\bm q|_{k+1}+|y|_{k+2}+|\bm p|_{k+1}+|z|_{k+2}),\\
		\|y-y_h\|_{\mathcal T_h}&\lesssim h^{k+2}(|\bm q|_{k+1}+|y|_{k+2}+|\bm p|_{k+1}+|z|_{k+2}),\\
		\|z-z_h\|_{\mathcal T_h}&\lesssim h^{k+2}(|\bm q|_{k+1}+|y|_{k+2}+|\bm p|_{k+1}+|z|_{k+2}).
		\end{align}
	\end{subequations}
\end{theorem}
\begin{proof}
	Recalling the continuous and discretized optimality conditions \eqref{eq_adeq_e} and \eqref{EDG_full_discrete_e} gives
	% From \eqref{eq_adeq_e}, \eqref{HDG_full_discrete_e} and \eqref{eq_uuh_yhuyh}, 
	\begin{align*}
	\hspace{3em}&\hspace{-3em}   \gamma\|u-u_h\|^2_{\mathcal T_h}+\|\zeta_{ y}\|^2_{\mathcal T_h}\\
	&=( z_h+\gamma u_h,u-u_h)_{\mathcal T_h}-( z_h(u)+\gamma u,u-u_h)_{\mathcal T_h}\\
	&=-( z_h(u)- z,u-u_h)_{\mathcal T_h}\\
	&\le \| z_h(u)- z\|_{\mathcal T_h} \|u-u_h\|_{\mathcal T_h}\\
	&\le\frac{1}{2\gamma}\| z_h(u)- z\|^2_{\mathcal T_h}+\frac{\gamma}{2}\|u-u_h\|^2_{\mathcal T_h}.
	\end{align*}
	By \Cref{lemma:step3_conv_rates}, we have
	\begin{align}\label{eqn:estimate_u_zeta_y}
	\|u-u_h\|_{\mathcal T_h}+\|\zeta_{ y}\|_{\mathcal T_h}&\lesssim h^{k+2}(|\bm q|_{k+1}+|y|_{k+2}+|\bm p|_{k+1}+|z|_{k+2}).
	\end{align}
	Then, by the triangle inequality and  \Cref{le} we obtain
	\begin{align*}
	\|y-y_h\|_{\mathcal T_h}&\lesssim h^{k+2}(|\bm q|_{k+1}+|y|_{k+2}+|\bm p|_{k+1}+|z|_{k+2}).
	\end{align*}
	Finally, since $z = \gamma u $ and $z_h = \gamma u_h$ we have
	\begin{align*}
	\|z-z_h\|_{\mathcal T_h}&\lesssim h^{k+2}(|\bm q|_{k+1}+|y|_{k+2}+|\bm p|_{k+1}+|z|_{k+2}).
	\end{align*}
\end{proof}

\subsubsection{Step 8: Estimate for $\|q-q_h\|_{\mathcal T_h}$ and  $\|p-p_h\|_{\mathcal T_h}$.}

\begin{lemma}
	We have
	\begin{subequations}
		\begin{align}
		\|\zeta_{\bm q}\|_{\mathcal T_h} &\lesssim h^{k+2}(|\bm q|_{k+1}+|y|_{k+2}+|\bm p|_{k+1}+|z|_{k+2}),\label{err_Lhu_qh}\\
		\|\zeta_{\bm p}\|_{\mathcal T_h} &\lesssim h^{k+2}(|\bm q|_{k+1}+|y|_{k+2}+|\bm p|_{k+1}+|z|_{k+2}).\label{err_Lhu_ph}
		\end{align}
	\end{subequations}
\end{lemma}
\begin{proof}
	By  \Cref{property_B}, the error equation \eqref{eq_yh_yhu}, and the estimate \eqref{eqn:estimate_u_zeta_y} we have
	\begin{align*}
	\|\zeta_{\bm q}\|^2_{\mathcal T_h} &\lesssim  \mathscr B_1(\zeta_{\bm q},\zeta_y,\zeta_{\widehat y};\zeta_{\bm q},\zeta_y,\zeta_{\widehat y})\\
	&=( u- u_h,\zeta_{ y})_{\mathcal T_h}\\
	&\le\| u- u_h\|_{\mathcal T_h}\|\zeta_{ y}\|_{\mathcal T_h}\\
	&\lesssim h^{2k+4}(|\bm q|_{k+1}+|y|_{k+2}+|\bm p|_{k+1}+|z|_{k+2})^2.
	\end{align*}
	Similarly,  by  \Cref{property_B}, the error equation \eqref{eq_zh_zhu},  \Cref{lemma:step3_conv_rates}, and  \Cref{thm:estimates_u_y_z} we have
	\begin{align*}
	\|\zeta_{\bm p}\|^2_{\mathcal T_h} &\lesssim  \mathscr B_2(\zeta_{\bm p},\zeta_z,\zeta_{\widehat z};\zeta_{\bm p},\zeta_z,\zeta_{\widehat z})\\
	&=-(\zeta_{ y},\zeta_{ z})_{\mathcal T_h}\\
	&\le\|\zeta_{y}\|_{\mathcal T_h}\|\zeta_{ z}\|_{\mathcal T_h}\\
	&\le\|\zeta_{y}\|_{\mathcal T_h} ( \| z_h(u) - z \|_{\mathcal T_h} + \| z - z_h \|_{\mathcal T_h} )\\
	&\lesssim h^{2k+4}(|\bm q|_{k+1}+|y|_{k+2}+|\bm p|_{k+1}+|z|_{k+2})^2.
	\end{align*}
\end{proof}
The above lemma along with the triangle inequality,  \Cref{le}, and  \Cref{lemma:step3_conv_rates} complete the proof of the main result:
\begin{theorem}
	We have
	\begin{subequations}
		\begin{align}
		\|\bm q-\bm q_h\|_{\mathcal T_h}&\lesssim h^{k+1}(|\bm q|_{k+1}+|y|_{k+2}+|\bm p|_{k+1}+|z|_{k+2}),\label{err_q}\\
		\|\bm p-\bm p_h\|_{\mathcal T_h}&\lesssim h^{k+1}(|\bm q|_{k+1}+|y|_{k+2}+|\bm p|_{k+1}+|z|_{k+2})\label{err_p}.
		\end{align}
	\end{subequations}
\end{theorem}

\section{Numerical Experiments}
\label{sec:numerics}

In this section, we present two numerical examples to confirm our theoretical results. We consider the problems on a square domain $\Omega = [0,1]\times [0,1] \subset \mathbb{R}^2$.  For the  two examples, we take $\gamma = 1$,  $\tau_1 = 1$, $\bm \beta = [x_2,x_1]$ and the exact state $y(x_1,x_2) = \sin(\pi x_1)$. We used the optimize-then-discretize (OD) approach in  \Cref{example1}  and the discretize-then-optimize (DO) approach in \Cref{example2}. In these examples, the data $f$, $ g $, and $y_d$ is generated from the optimality system \eqref{eq_adeq}  after we specified the exact dual state $ z(x_1,x_2) = \sin(\pi x_1)\sin(\pi x_2)$.

Numerical results for $ k = 0 $ and $ k = 1 $ for the two approaches are shown in Table \ref{table_1}--Table \ref{table_4}.  The observed convergence rates and numerical results exactly match the theoretical results.

\begin{example}\label{example1}
	For the OD approach, we set the stabilization parameter $ \tau_2 $ using \textbf{(A1)}; hence, conditions \textbf{(A1)}-\textbf{(A2)} are satisfied. We obtain optimal convergence rates for all variables for $k=0$ and $k=1$ in \Cref{table_1} and \Cref{table_2}, respectively. This matches our theoretical results.
	\begin{table}%[!hbp]
		\begin{center}
			\begin{tabular}{|c|c|c|c|c|c|}
				\hline
				$h/\sqrt 2$ &$1/8$& $1/16$&$1/32$ &$1/64$ & $1/128$ \\
				\hline
				$\norm{\bm{q}-\bm{q}_h}_{0,\Omega}$& 2.8775E-01   &1.4501E-01   &7.2649E-02   &3.6342E-02   &1.8173E-02 \\
				\hline
				order&-& 0.98861   &0.99716   &0.99929   &0.99982\\
				\hline
				$\norm{\bm{p}-\bm{p}_h}_{0,\Omega}$& 2.1036E-01   &1.0341E-01   &5.1480E-02   &2.5712E-02   &1.2852E-02\\
				\hline
				order&-& 1.0244   &1.0063   &1.0016   &1.0004\\
				\hline
				$\norm{{y}-{y}_h}_{0,\Omega}$&1.1842E-02   &3.2095E-03   &8.4824E-04   &2.1887E-04   &5.5641E-05 \\
				\hline
				order&-& 1.8834   &1.9198   &1.9544   &1.9759 \\
				\hline
				$\norm{{z}-{z}_h}_{0,\Omega}$& 1.8304E-02   &5.3420E-03   &1.4422E-03   &3.7460E-04   &9.5451E-05
				 \\
				\hline
				order&-& 1.7767   &1.8891   &1.9449   &1.9725\\
				\hline
			\end{tabular}
		\end{center}
		\caption{Example \ref{example1}: Errors for the state $y$, adjoint state $z$, and the fluxes $\bm q$ and $\bm p$ when $k=0$ with the OD approach.}\label{table_1}
	\end{table}

	\begin{table}%[!hbp]
		\begin{center}
			\begin{tabular}{|c|c|c|c|c|c|}
				\hline
				$h/\sqrt 2$ &$1/8$& $1/16$&$1/32$ &$1/64$ & $1/128$ \\
				\hline
				$\norm{\bm{q}-\bm{q}_h}_{0,\Omega}$&1.8365E-02   &4.9165E-03   &1.2726E-03   &3.2189E-04   &8.0742E-05  \\
				\hline
				order&-& 1.9012   &1.9498   &1.9831   &1.9952\\
				\hline
				$\norm{\bm{p}-\bm{p}_h}_{0,\Omega}$& 1.6649E-02   &5.6050E-03   &1.5952E-03   &4.1463E-04   &1.0475E-04 \\
				\hline
				order&-&  1.5707  &1.8129   &1.9439   &1.9848\\
				\hline
				$\norm{{y}-{y}_h}_{0,\Omega}$&1.3524E-03   &1.8347E-04   &2.3956E-05   &3.0691E-06   &3.8882E-07\\
				\hline
				order&-& 2.8819  &2.9371   &2.9645   &2.9807 \\
				\hline
				$\norm{{z}-{z}_h}_{0,\Omega}$& 3.2125E-03   &4.2489E-04   &5.4721E-05   &6.9745E-06   &8.8190E-07
				 \\
				\hline
				order&-& 2.9186   &2.9569   &2.9719  & 2.9834 \\
				\hline
			\end{tabular}
		\end{center}
		\caption{Example \ref{example1}: Errors for the state $y$, adjoint state $z$, and the fluxes $\bm q$ and $\bm p$ when $k=1$ with the OD approach.}\label{table_2}
	\end{table}
\end{example}

\begin{example}\label{example2}
	For the DO approach, we used the same data as in \Cref{example1}. From the tables we can see that the numerical results are exactly the same with the OD approach, which confirms our theoretical results.
	\begin{table}%[!hbp]
		\begin{center}
			\begin{tabular}{|c|c|c|c|c|c|}
				\hline
				$h/\sqrt 2$ &$1/8$& $1/16$&$1/32$ &$1/64$ & $1/128$ \\
				\hline
				$\norm{\bm{q}-\bm{q}_h}_{0,\Omega}$& 2.8775E-01   &1.4501E-01   &7.2649E-02   &3.6342E-02   &1.8173E-02 \\
				\hline
				order&-& 0.98861   &0.99716   &0.99929   &0.99982\\
				\hline
				$\norm{\bm{p}-\bm{p}_h}_{0,\Omega}$& 2.1036E-01   &1.0341E-01   &5.1480E-02   &2.5712E-02   &1.2852E-02\\
				\hline
				order&-& 1.0244   &1.0063   &1.0016   &1.0004\\
				\hline
				$\norm{{y}-{y}_h}_{0,\Omega}$&1.1842E-02   &3.2095E-03   &8.4824E-04   &2.1887E-04   &5.5641E-05 \\
				\hline
				order&-& 1.8834   &1.9198   &1.9544   &1.9759 \\
				\hline
				$\norm{{z}-{z}_h}_{0,\Omega}$& 1.8304E-02   &5.3420E-03   &1.4422E-03   &3.7460E-04   &9.5451E-05
				\\
				\hline
				order&-& 1.7767   &1.8891   &1.9449   &1.9725\\
				\hline
			\end{tabular}
		\end{center}
		\caption{Example \ref{example2}: Errors for the state $y$, adjoint state $z$, and the fluxes $\bm q$ and $\bm p$ when $k=0$ with the DO approach.}\label{table_3}
	\end{table}

	\begin{table}%[!hbp]
		\begin{center}
			\begin{tabular}{|c|c|c|c|c|c|}
				\hline
				$h/\sqrt 2$ &$1/8$& $1/16$&$1/32$ &$1/64$ & $1/128$ \\
				\hline
				$\norm{\bm{q}-\bm{q}_h}_{0,\Omega}$&1.8365E-02   &4.9165E-03   &1.2726E-03   &3.2189E-04   &8.0742E-05  \\
				\hline
				order&-& 1.9012   &1.9498   &1.9831   &1.9952\\
				\hline
				$\norm{\bm{p}-\bm{p}_h}_{0,\Omega}$& 1.6649E-02   &5.6050E-03   &1.5952E-03   &4.1463E-04   &1.0475E-04 \\
				\hline
				order&-&  1.5707  &1.8129   &1.9439   &1.9848\\
				\hline
				$\norm{{y}-{y}_h}_{0,\Omega}$&1.3524E-03   &1.8347E-04   &2.3956E-05   &3.0691E-06   &3.8882E-07\\
				\hline
				order&-& 2.8819  &2.9371   &2.9645   &2.9807 \\
				\hline
				$\norm{{z}-{z}_h}_{0,\Omega}$& 3.2125E-03   &4.2489E-04   &5.4721E-05   &6.9745E-06   &8.8190E-07
				\\
				\hline
				order&-& 2.9186   &2.9569   &2.9719  & 2.9834 \\
				\hline
			\end{tabular}
		\end{center}
		\caption{Example \ref{example2}: Errors for the state $y$, adjoint state $z$, and the fluxes $\bm q$ and $\bm p$ when $k=1$ with the DO approach.}\label{table_4}
	\end{table}
\end{example}

\section{Conclusions}
We considered a recently proposed EDG method to approximate the solution of an optimal distributed control problems for an elliptic convection diffusion equation.  We showed the optimize-then-discretize and discretize-then-optimize approaches coincide, and proved optimal a priori error estimates for the control, state, dual state, and their fluxes.  EDG methods are known to be competitive for convection dominated problems; therefore, this new EDG method has potential for optimal control problems involving such PDEs.

\section*{Acknowledgements}   X.~Zhang thanks Missouri University of Science and Technology for hosting him as a visiting scholar; some of this work was completed during his research visit. J.~Singler and Y.~Zhang were supported in part by National Science Foundation grant DMS-1217122.  J.~Singler and Y.~Zhang thank the IMA for funding research visits, during which some of this work was completed.

\section{Appendix}

By simple algebraic operations in equation \eqref{large_op_b}, we obtain the following formulas for $ G_1 $, $ G_2 $, $ G_3$, $ G_4 $, $ H_1 $, and $ H_2 $ in \eqref{local_elei}:
\begin{align*}
G_1 &= -A_1^{-1}A_2(A_4+A_2^TA_1^{-1}A_2)^{-1}(A_5-A_2^TA_1^{-1}A_3)-A_1^{-1}A_3,\\
G_2 &= A_1^{-1}A_2(A_4+A_2^TA_1^{-1}A_2)^{-1}A_6,\\
G_3 &= -(A_4+A_2^TA_1^{-1}A_2)^{-1}(A_5-A_2^TA_1^{-1}A_3),\\
G_4 &= (A_4+A_2^TA_1^{-1}A_2)^{-1}A_6,\\
H_1 &= A_1^{-1}A_2(A_4+A_2^TA_1^{-1}A_2)^{-1}(b_3 - b_4+A_2^TA_1^{-1}b_2 ) - A_1^{-1}b_2,\\
H_2 &=(A_4+A_2^TA_1^{-1}A_2)^{-1}(b_3 - b_4+A_2^TA_1^{-1}b_2 ).
\end{align*}
In general, forming these quantities is impractical; however, for the EDG method described in this work these matrices can be easily computed.  We briefly sketch this process below.

% but here we can compute in this way, this is one of  advantages of HDG method.

Since the spaces $ \bm{V}_h $ and $ W_h $ consist of discontinuous polynomials, some of the system matrices are block diagonal and each block is small and symmetric positive definite (SSPD).  The inverse of such a matrix is another matrix of the same type, and the inverse is easily computed by inverting each small block.  Furthermore, the inverse of each small block can be computed in parallel.

It can be checked that $ A_1 $ is a SSPD block diagonal matrix, and therefore $ A_1^{-1} $ is easily computed and is also a SSPD block diagonal matrix.  Therefore, $ G_1 $, $ G_2 $, $ G_3 $, $ G_4 $, $ H_1 $, and $ H_2 $ are easily computed since $ A_4 + A_2^T A_1^{-1} A_2 $ is also a SSPD block diagonal matrix.  Also, once these quantities are computed, $ G_5 $, $ G_6 $, and $ H_3 $ in \eqref{local_elei} are also easy to compute using \eqref{large_op_b}.

%\section*{References}
\bibliographystyle{plain}
\bibliography{yangwen_ref_papers,yangwen_ref_books}

\end{document}